\documentclass[letterpaper, 11pt,  reqno]{amsart}

\usepackage{amsmath,amssymb,amscd,amsthm,amsxtra, esint}

\usepackage{tikz} 

\usepackage{comment}
\usepackage{mathtools}
\usepackage{color}
\usepackage[implicit=true]{hyperref}
\usepackage{scalerel} 

\usepackage{cases}

\usepackage[left=32mm, right=32mm, 
bottom=27mm]{geometry}

\allowdisplaybreaks[2]

\usepackage{color}

\definecolor{gr}{rgb}   {0.,   0.69,   0.23 }
\definecolor{bl}{rgb}   {0.,   0.5,   1. }
\definecolor{mg}{rgb}   {0.85,  0.,    0.85}
\definecolor{yl}{rgb}   {0.8,  0.7,   0.}
\definecolor{or}{rgb}  {0.7,0.2,0.2}

\usetikzlibrary{shapes.misc}
\usetikzlibrary{shapes.symbols}
\usetikzlibrary{shapes.geometric}

\tikzset{
	ddot/.style={circle,fill=white,draw=black,inner sep=0pt,minimum size=0.8mm},
	>=stealth,
	}

\tikzset{
	ddot2/.style={circle,fill=black,draw=black,inner sep=0pt,minimum size=0.8mm},
	>=stealth,
	}

\newtheorem{theorem}{Theorem} [section]

\newtheorem{lemma}[theorem]{Lemma}
\newtheorem{proposition}[theorem]{Proposition}
\newtheorem{remark}[theorem]{Remark}

\DeclareMathOperator*{\supp}{supp}

\newcommand{\I}{\mathcal{I}}

\newcommand{\noi}{\noindent}
\newcommand{\Z}{\mathbb{Z}}
\newcommand{\R}{\mathbb{R}}
\newcommand{\C}{\mathbb{C}}
\newcommand{\T}{\mathbb{T}}

\newcommand{\E}{\mathbb{E}}

\newcommand{\F}{\mathcal{F}}

\newcommand{\al}{\alpha}

\newcommand{\dl}{\delta}

\newcommand{\Dl}{\Delta}
\newcommand{\eps}{\varepsilon}

\newcommand{\g}{\gamma}
\newcommand{\G}{\Gamma}

\newcommand{\s}{\sigma}

\newcommand{\ft}{\widehat}

\newcommand{\wt}{\widetilde}
\newcommand{\cj}{\overline}

\newcommand{\dt}{\partial_t}

\newcommand{\ta}{\theta}

\renewcommand{\l}{\ell}
\renewcommand{\o}{\omega}
\renewcommand{\O}{\Omega}

\newcommand{\les}{\lesssim}
\newcommand{\ges}{\gtrsim}

\newcommand{\jb}[1]
{\langle #1 \rangle}

\newcommand{\ind}{\mathbf 1}

\newcommand{\N}{\mathbb{N}}

\renewcommand{\H}{\mathcal{H}}

\newtheorem*{ackno}{Acknowledgements}

\numberwithin{equation}{section}
\numberwithin{theorem}{section}

\usepackage[english, french]{babel}

\makeatletter
\@namedef{subjclassname@2020}{%
  \textup{2020} Mathematics Subject Classification}
\makeatother

\begin{document}
\baselineskip = 14pt

\selectlanguage{english}

\title[Probabilistic well-posedness of the 2D quadratic NLS]{On the probabilistic well-posedness of the two-dimensional periodic nonlinear Schr\"odinger equation with the quadratic nonlinearity $|u|^2$}

\author[R.~Liu]{Ruoyuan Liu}

\address{Ruoyuan Liu,  School of Mathematics\\
The University of Edinburgh\\
and The Maxwell Institute for the Mathematical Sciences\\
James Clerk Maxwell Building\\
The King's Buildings\\
Peter Guthrie Tait Road\\
Edinburgh\\ 
EH9 3FD\\
 United Kingdom}

\email{ruoyuan.liu@ed.ac.uk}



\maketitle

\vspace{-10mm}

\begin{abstract}
We study the two-dimensional periodic nonlinear Schr\"odinger equation (NLS) with the quadratic nonlinearity $|u|^2$. In particular, we study the quadratic NLS with random initial data distributed according to a fractional derivative (of order $\al \geq 0$) of the Gaussian free field. After removing the singularity at the zeroth frequency, we prove that the quadratic NLS is almost surely locally well-posed for $\al < \frac 12$ and is probabilistically ill-posed for $\al \geq \frac 34$ in a suitable sense. The probabilistic ill-posedness result shows that in the case of rough random initial data and a quadratic nonlinearity, the standard probabilistic well-posedness theory for NLS breaks down before reaching the critical value $\al = 1$ predicted by the scaling analysis due to Deng, Nahmod, and Yue (2019), and thus this paper is a continuation of the work by Oh and Okamoto (2021) on stochastic nonlinear wave and heat equations by building an analogue for NLS.
\end{abstract}

\begin{otherlanguage}{french}
\begin{abstract}

Nous \'etudions l'\'equation de Schr\"odinger  non lin\'eaire (NLS) 
avec la non-lin\'earit\'e quadratique $|u|^2$
sur  un tore 
de dimension deux.
En particulier, nous \'etudions  NLS quadratique avec 
une donn\'ee  initiale al\'eatoire distribu\'ee selon une d\'eriv\'ee fractionnaire (d'ordre $ \al \geq 0$) du 
champ libre gaussien. 
Apr\`es suppression de la singularit\'e \`a la fr\'equence z\'ero, nous prouvons que  NLS quadratique est presque s\^urement localement bien pos\'e pour $\al < \frac 12$ et est 
 mal pos\'e pour $\al \geq \frac 34$ dans un sens probabiliste appropri\'e. 
 Le fait  que NLS quadratique soit mal pos\'e dans un  sens probabiliste
 montre que dans le cas de donn\'ees initiales al\'eatoires 
 \`a basse regularit\'e
  et d'une non-lin\'earit\'e quadratique, 
 la th\'eorie de Cauchy probabiliste standard 
 pour NLS 
 perd sa 
validit\'e
 avant d'atteindre 
 le valeur critique $\al = 1$ pr\'edite par l'analyse 
 due \`a Deng, Nahmod et Yue (2019).
Cet article est donc une continuation, dans le cas de NLS,  des travaux de Oh et Okamoto (2021) sur les \'equations stochastiques non lin\'eaires 
des ondes et du chaleur.

\end{abstract}
\end{otherlanguage}

\tableofcontents

\section{Introduction}

\subsection{Quadratic NLS with random initial data}
\label{SUBSEC:ini}

We consider the Cauchy problem for the following quadratic nonlinear Schr\"odinger equation (NLS) on the two-dimensional torus $\T^2 = (\R / 2\pi \Z)^2$:
\begin{equation}
\begin{cases}
i \dt u + \Dl u = |u|^2 - \fint |u|^2 \\
u|_{t=0} = u_0^\o.
\end{cases}
\label{qNLS}
\end{equation}

\noi
Here, $\fint f(x) dx := \tfrac{1}{(2 \pi)^2} \int_{\T^2} f(x) dx$ and $u_0^\o$ is the following Gaussian random initial data:
\begin{align}
u_0^\o (x) = \sum_{n \in \Z^2} \frac{g_n(\o)}{ \jb{n}^{1 - \al} } e^{i n \cdot x},
\label{ini}
\end{align}

\noi
where $\al \in \R$ and $\{g_n\}_{n \in \Z^2}$ is a set of independent standard complex-valued Gaussian random variables with $\E g_n = 0$ and $\E |g_n|^2 = 1$. Note that when $\al = 0$, $u_0^\o$ is the Gaussian random initial data distributed according to the massive Gaussian free field on $H^s (\T^2)$, $s < 0$.

Over the past several decades, we have witnessed tremendous progress on well-posedness issues of NLS with various types of nonlinearities from both deterministic and probabilistic points of views. Let us first briefly mention the deterministic well-posedness results for NLS on periodic domains. In \cite{Bour93}, Bourgain introduced the Fourier restriction norm method (see Subsection \ref{SUBSEC:Xsb}) and proved NLS with a gauge-invariant nonlinearity in the low regularity setting. In particular, he showed local well-posedness of the cubic NLS (i.e.~with nonlinearity $|u|^2 u$) in $H^s (\T^2)$ for any $s > 0$ by proving the following $L^4$-Strichartz estimate on $\T^2$ (See also Lemma \ref{LEM:L4}):
\begin{align}
\| e^{it \Dl} u \|_{L^4([-1, 1]; L^4(\T^2))} \les \| u \|_{H^s (\T^2)},
\label{L4d}
\end{align}

\noi
for any $s > 0$. We now focus on the following quadratic NLS:
\begin{align}
i \dt u + \Dl u = |u|^2.
\label{qNLSd}
\end{align}

\noi
For \eqref{qNLSd}, one can easily obtain local well-posedness in $H^s (\T^2)$ for $s > 0$ by using the $L^3$-Strichartz estimate with a derivative loss, which follows from interpolating \eqref{L4d} and the trivial $L^2$ bound. In \cite{Kish19}, Kishimoto proved ill-posedness of \eqref{qNLSd} in $H^s(\T^2)$ for $s < 0$. In a recent preprint \cite{LO}, the author and Oh proved local well-posedness of \eqref{qNLSd} in $H^0(\T^2) = L^2 (\T^2)$, thus completing the deterministic well-posedness theory of \eqref{qNLSd}.

We now turn our attention to NLS with rough random initial data. The idea of constructing local-in-time solutions of NLS using random initial data was first introduced by Bourgain in \cite{Bour96}, where he proved almost sure local well-posedness of the (renormalized) cubic NLS on $\T^2$ with random initial data \eqref{ini} with $\al = 0$. See also \cite{Bour97, CO, DNY2, DNY3, FOSW} for more results on almost sure local well-posedness of NLS with various types of nonlinearities on periodic domains with random initial data of the form \eqref{ini}. The almost sure local well-posedness results of NLS with a quadratic nonlinearity, to the best of the author's knowledge, have not been explored yet. In this paper, we choose to work with the quadratic NLS \eqref{qNLS} (see Remark \ref{RMK:zero} below for the necessity of removing the mean of the nonlinearity). Note that the initial data $u_0^\o$ almost surely belongs to $H^{-\al - \eps} (\T^2) \setminus H^{- \al} (\T^2)$ for any $\eps > 0$. See Lemma B.1 in \cite{BT08}. When $\al < 0$, the initial data $u_0^\o$ almost surely belongs to $H^s (\T^2)$ for sufficiently small $s = s(\al) > 0$, so that we can easily prove almost sure local well-posedness of \eqref{qNLS} by using the $L^3$-Strichartz estimate mentioned above. Our goal in this paper is to (i) obtain probabilistic local well-posedness of \eqref{qNLS} with $\al \geq 0$ and (ii) identify bad behaviors of \eqref{qNLS} when $\al$ gets too large. Specifically, we show that \eqref{qNLS} is almost surely locally well-posed when $0 \leq \al < \frac 12$ (see Subsection \ref{SUBSEC:well} below) and is probabilistically ill-posed in a suitable sense when $\al \geq \frac 34$ (see Subsection \ref{SUBSEC:ill} below).

\subsection{Almost sure local well-posedness of the quadratic NLS}
\label{SUBSEC:well}

In this subsection, we state our almost sure local well-posedness theorem for the quadratic NLS \eqref{qNLS} and describe our strategy for proving our result. We define 
\begin{align}
z(t) := z^\o (t) = e^{it \Dl} u_0^\o = \sum_{n \in \Z^2} \frac{g_n (\o)}{\jb{n}^{1 - \al}} e^{-it |n|^2 + in \cdot x}
\label{defz}
\end{align}

\noi
as the solution to the linear Schr\"odinger equation with the random initial data $u_0^\o$:
\begin{align*}
\begin{cases}
i \dt z + \Dl z = 0 \\
z|_{t = 0} = u_0^\o.
\end{cases}
\end{align*}

\noi
The precise statement of our almost sure local well-posedness result reads as follows.

\begin{theorem}
\label{THM:LWP}
Let $0 \leq \al < \frac 12$ and $s > 0$. Then, the quadratic NLS \eqref{qNLS} is almost surely locally well-posed in the class $z + C([-T, T]; H^s (\T^2))$. More precisely, there exist $T_0 > 0$ and constants $C, c, \ta > 0$ such that for all $0 < T \leq T_0$, there exists a set $\O_T \subset \O$ with the following properties:
\begin{itemize}
\item[(1)] $P(\O \setminus \O_T) \leq C \exp(- \frac{c}{T^\ta})$.

\item[(2)] For each $\o \in \O_T$, there exists a unique solution $u = u^\o$ to \eqref{qNLS} on $[-T, T]$ with $u|_{t = 0} = u_0^\o$ in the class $z + C([-T, T]; H^s (\T^2))$.
\end{itemize}
\end{theorem}

We prove Theorem \ref{THM:LWP} by using the following first order expansion \cite{McKean, Bour96, DPD}:
\begin{align}
u = z + v,
\label{exp}
\end{align}

\noi
where the residual term $v$ satisfies the following equation:
\begin{align}
\begin{cases}
i \dt v + \Dl v = |z + v|^2 - \fint |z + v|^2 \\
v|_{t = 0} = 0.
\end{cases}
\label{vNLS}
\end{align}

\noi
The uniqueness statement in Theorem \ref{THM:LWP} refers to the uniqueness of $v$ as a solution to this perturbed quadratic NLS \eqref{vNLS} in an appropriate space (see the $X_T^{s, b}$-spaces in Subsection \ref{SUBSEC:Xsb}).

The well-posedness result of the perturbed quadratic NLS \eqref{vNLS} follows from the corresponding bilinear estimates to the quadratic terms $|v|^2$, $v \cj{z}$, $z \cj{v}$, and $|z|^2$. To prove these bilinear estimates, we use the operator norm approach based on the random tensor theory (see Subsection \ref{SUBSEC:rand}) developed by Deng, Nahmod, and Yue in \cite{DNY3}. See \cite{DNY3, Bri, Seo, OWZ, BDNY} for some applications of the random tensor theory in the study of well-posedness of random dispersive equations. We present the details of the corresponding bilinear estimates in Section \ref{SEC:bilin} and the proof of Theorem \ref{THM:LWP} in Section \ref{SEC:well}.

\begin{remark} \rm
\label{RMK:zero}

Let us consider the following quadratic NLS:
\begin{align}
\begin{cases}
i \dt u + \Dl u = |u|^2 \\
u|_{t = 0} = u_0^\o,
\end{cases}
\label{qNLSinit}
\end{align}

\noi
where $u_0^\o$ is the Gaussian random initial data as defined in \eqref{ini}. For $N \in \N$, we let $u_{0, N}^\o$ be the truncation of $u_0^\o$ as defined in \eqref{defz} to frequencies $\{ |n| \leq N \}$, and we define $z_N (t) := e^{it \Dl} u_{0, N}^\o$. One can easily check that the zeroth frequency of the following Picard second iterate
\begin{align*}
\int_0^t e^{i (t - t') \Dl} \big( |z_N (t')|^2 \big) dt'
\end{align*}

\noi
diverges almost surely when $\al \geq 0$ (see, for example, Subsection 4.4 in \cite{OO}). Thus, in order to make the almost sure local well-posedness problem non-trivial, we need to remove this singular behavior of \eqref{qNLSinit} occurring at the zeroth frequency.

A more natural way of dealing with the above issue is to introduce the following renormalized quadratic NLS:
\begin{align}
\begin{cases}
i \dt u_N + \Dl u_N = |u_N|^2 - \s_N \\
u_N|_{t = 0} = u_{0, N}^\o,
\end{cases}
\label{qNLSr}
\end{align}

\noi
where $\s_N = \E[|u_{0, N}^\o|^2]$. However, due to the lack of the conservation of mass $\int |u|^2$, there seems to be no easy way to establish the equivalence between \eqref{qNLSr} and the quadratic NLS \eqref{qNLS}. One can compare this situation with the cubic NLS on $\T^2$ in \cite{Bour96}, where Bourgain used a Gauge transform $u_N = \exp (2i (\fint |u_N|^2 - \s_N) t) \cdot v_N$ to show the equivalence between
\begin{align*}
i \dt u_N + \Dl u_N = |u_N|^2 u_N - 2 \s_N u_N
\end{align*}
and
\begin{align*}
i \dt v_N + \Dl v_N = \bigg( |v_N|^2 - 2 \fint |v_N|^2 \bigg) v_N.
\end{align*}

\noi
Here in the case of the cubic NLS, the quantity $\fint |u_N|^2 - \s_N$ is time invariant and one can easily recover $u_N$ from $v_N$ by noticing that $\fint |u_N|^2 = \fint |v_N|^2$. Nevertheless, similar transforms do not seem to apply to the case of the renormalized quadratic NLS \eqref{qNLSr}.

One of the problem with directly proceeding with \eqref{qNLSr} is that, by using the first order expansion $u_N = z_N + v_N$ with $z_N = e^{it \Dl} u_{0, N}^\o$ and letting $\I$ be the Duhamel operator, the zeroth frequencies of the bilinear terms $\I (v_N \cj{z_N})$, $\I (z_N \cj{v_N})$
, and $\I (z_N \cj{z_N})$ cannot be shown to converge when $\al \geq 0$ using our approach. Another problem is that the remainder term $v_N$ is not necessarily of mean zero, which causes a trouble in estimating the bilinear term $\I (z_N \cj{v_N})$ when $\al \geq 0$. See Proposition \ref{PROP:zz} and Remark \ref{RMK:zv} for more details. Therefore, in this paper, we choose to focus on the quadratic NLS \eqref{qNLS} (i.e. with nonlinearity $|u|^2 - \fint |u|^2$).

\end{remark}

\begin{remark} \rm
\label{RMK:mol}
Let $\eta \in C(\R^2; [0, 1])$ be a mollification kernel such that $\int \eta \, dx = 1$ and $\supp \eta \subset (-1, 1]^2 \simeq \T^2$. For $0 < \eps \leq 1$, we define $\eta_\eps (x) = \eps^{-2} \eta (\eps^{-1} x)$, so that $\{ \eta_\eps \}_{0 < \eps \leq 1}$ forms an approximate identity on $\T^2$. With a slight modification of the proof of Theorem \ref{THM:LWP}, we can show that when $\al < \frac 12$, the solution $u_\eps$ to
\begin{align*}
\begin{cases}
i \dt u_\eps + \Dl u_\eps = |u_\eps|^2 - \fint |u_\eps|^2 \\
u_\eps |_{t = 0} = \eta_\eps * u_0^\o
\end{cases}
\end{align*}

\noi
converges in probability to some (unique) limiting distribution $u$ in $C([-T_\o, T_\o]; H^{-\al -} (\T^2))$ with $T_\o > 0$ almost surely. Here, the limiting distribution $u$ is independent of the choice of the mollification kernel $\eta$.
\end{remark}

\begin{remark} \rm
Let us also consider probabilistic well-posedness of NLS with other quadratic nonlinearities:
\begin{align}
\begin{cases}
i \dt u + \Dl u = \mathcal{N} (u) \\
u|_{t = 0} = u_0^\o
\end{cases}
\label{qNLS2}
\end{align}

\noi
with $\mathcal{N} (u) = u^2$ or $\cj{u}^2$ and $u_0^\o$ as defined in \eqref{ini}. We first point out that these nonlinearities have different corresponding phase functions: $- 2 n \cdot n_2$ for $|u|^2$, $- 2 n_1 \cdot n_2$ for $u^2$, and $|n|^2 + |n_1|^2 + |n_2|^2$ for $\cj{u}^2$. Here, $n_1$ corresponds to the frequency of the first incoming wave, $n_2$ corresponds to the frequency of the second incoming wave, and $n$ corresponds to the frequency of the outgoing wave.

For $\mathcal{N} (u) = u^2$, we can use a similar argument as in the proof of Theorem \ref{THM:LWP} to obtain almost sure local well-posedness of \eqref{qNLS2} when $\al < \frac 12$. We point out that in this case, we do not need to remove any singularities as compared to the case of $\mathcal{N} (u) = |u|^2$.

For $\mathcal{N} (u) = \cj{u}^2$, due to the different nature of the corresponding phase function, we expect that one can go beyond the range $\al < \frac 12$ established for the almost sure local well-posedness for NLS with nonlinearities $|u|^2$ and $u^2$. However, the method for proving Theorem \ref{THM:LWP} based on the first order expansion is not enough for this purpose, since the corresponding bilinear estimate involving the product of two random linear solutions (Proposition \ref{PROP:zz} (iii)) is still only valid when $\al < \frac 12$. In this case, it may be possible to establish almost sure local well-posedness for some range of $\al \geq \frac 12$ using higher order expansions as in \cite{BOP3, OPTz} \footnote{This includes the paracontrolled approach used in \cite{GKO2, Bri, BDNY}}.
\end{remark}

\subsection{Probabilistic ill-posedness of the quadratic NLS}
\label{SUBSEC:ill}

In this subsection, we discuss probabilistic ill-posedness issues of the quadratic NLS \eqref{qNLS} for large values of $\al$. Given $N \in \N$, consider the following Picard second iterate:
\begin{align}
z_N^{(2)} (t) = \int_0^t e^{i (t - t') \Dl} \bigg( | z_N (t') |^2 - \fint |z_N (t')|^2 \bigg) \,dt',
\label{defzN2}
\end{align}

\noi
where 
\begin{align*}
z_N (t) = \sum_{\substack{n \in \Z^2 \\ |n| \leq N}} \frac{g_n (\o)}{\jb{n}^{1 - \al}} e^{-it |n|^2 + in \cdot x}
\end{align*}
is the truncation of the random linear solution $z$ as defined in \eqref{defz} to frequencies $\{ |n| \leq N \}$. We now state the following proposition regarding the non-convergence of every non-zero Fourier coefficient of the Picard second iterate $z_N^{(2)}$.

\begin{proposition}
\label{PROP:ill}
Let $n \neq 0$ and $t \neq 0$. For $\al \geq \frac 34$, the sequence $\big\{ \E \big[ | \F_x z_N^{(2)} (t, n) |^2 \big] \big\}_{N \in \N}$ goes to infinity as $N \to \infty$. Consequently, any subsequence of the sequence of random variables $\{ \F_x z_N^{(2)} (t, n) \}_{N \in \N}$ is not tight.
\end{proposition}

See Section \ref{SEC:ill} for the proof of Proposition \ref{PROP:ill}.

Proposition \ref{PROP:ill} implies that when $\al \geq \frac 34$, for every $n \neq 0$, any subsequence of $\{ \F_x z_N^{(2)} (t, n) \}_{N \in \N}$ does not converge in law. This in particular implies that standard methods for establishing almost sure local well-posedness such as the first order expansion \cite{McKean, Bour96, DPD} or its higher order variants \cite{BOP3, OPTz, GKO2, Bri, BDNY} do not work for $\al \geq \frac 34$.

Bearing in mind the above discussion, we now briefly discuss the probabilistic scaling and the associated critical regularity introduced by Deng, Nahmod, and Yue in \cite{DNY2}. The  notion of this probabilistic scaling is based on the observation that, if one wants to obtain local well-posedness, the Picard second iterate should not be rougher than the random linear solution. In \cite{DNY2}, Deng, Nahmod, and Yue provided heuristics for one to compute the probabilistic scaling critical regularity without too much difficulty, and they conjectured in the paper that for NLS with nonlinearities $|u|^{p-1} u$ ($p \in 2\N + 1$), almost sure local well-posedness should hold for all subcritical regularities. Indeed, in \cite{DNY3}, Deng, Nahmod, and Yue proved almost sure local well-posedness for NLS with nonlinearity $|u|^{p-1} u$ ($p \in 2\N + 1$) on $\T^d$ ($d \in \N$) in the full subcritical range relative to the probabilistic scaling. We point out that for NLS with the quadratic nonlinearity $|u|^2$, however, the probabilistic scaling does not seem to provide a useful prediction for probabilistic well-posedness issues, as we shall see in the following.

Let us compute the probabilistic scaling critical regularity for the quadratic NLS with nonlinearity $|u|^2$. Let $u_0^\o$ be the random initial data as defined in \eqref{ini}. Let $N \in \N$ be a dyadic number and consider the initial data $u_0^\o$ supported on frequencies $\{ |n| \sim N \}$:
\begin{align*}
P_N u_0^\o = \sum_{\substack{n \in \Z^2 \\ |n| \sim N}} \frac{g_n (\o)}{\jb{n}^{1 - \al}} e^{i n \cdot x}.
\end{align*}

\noi
Note that $\| P_N u_0^\o \|_{H^{- \al} (\T^2)} \sim 1$. Consider the following Picard second iterate term:\footnote{Here, we do not need to subtract the zeroth frequency of the nonlinearity since later on we only focus on the case when $|n| \sim N$.}
\begin{align*}
u_N^{(2)} (t) = \int_0^t e^{i(t - t') \Dl} \big( | e^{i t' \Dl} P_N u_0^\o |^2 \big) \, dt',
\end{align*}

\noi
whose $n$th Fourier coefficient can be computed as
\begin{align*}
\F_x u_N^{(2)} (t, n) = \int_0^t e^{-i t |n|^2} \sum_{\substack{n_1, n_2 \in \Z^2 \\ n_1 - n_2 = n \\ |n_1| \sim N, |n_2| \sim N }} e^{it' (|n|^2 - |n_1|^2 + |n_2|^2)} \frac{g_{n_1} (\o) \cj{g_{n_2}} (\o)}{ \jb{n_1}^{1 - \al} \jb{n_2}^{1 - \al} } \,dt'.
\end{align*}

\noi
We restrict our attention to the frequency range $\{ |n| \sim N \}$ of $u_N^{(2)} (t)$. Thus, by the Wiener chaos estimate (Lemma \ref{LEM:wie} below along with Chebyshev's inequality) and a counting estimate (see Lemma \ref{LEM:count1} (i) below), we can estimate the $H^{- \al} (\T^2)$-norm of $u_N^{(2)} (t)$ as follows:
\begin{align}
\| u_N^{(2)} (t) \|_{H^{- \al} (\T^2)}^2  &\les_t \sum_{\substack{n \in \Z^2 \\ |n| \sim N}} \jb{n}^{-2 \al} \bigg( \sum_{\substack{n_1, n_2 \in \Z^2 \\ n_1 - n_2 = n \\ |n_1| \sim N, |n_2| \sim N}} \frac{ g_{n_1} (\o) \cj{g_{n_2}} (\o)} {\jb{|n|^2 - |n_1|^2 + |n_2|^2} \jb{n_1}^{1 - \al} \jb{n_2}^{1 - \al}} \bigg)^2  \nonumber \\
&\les C_\o \sum_{\substack{n, n_1, n_2 \in \Z^2 \nonumber \\ n_1 - n_2 = n \\ |n|, |n_1|, |n_2| \sim N}} \frac{\jb{n}^{-2 \al}}{ \jb{|n|^2 - |n_1|^2 + |n_2|^2}^2 \jb{n_1}^{2 - 2 \al} \jb{n_2}^{2 - 2 \al} } \nonumber \\
&\sim C_\o N^{2 \al - 4}  \sum_{\substack{n, n_1, n_2 \in \Z^2 \\ n_1 - n_2 = n \\ |n|, |n_1|, |n_2| \sim N}} \frac{1}{\jb{|n|^2 - |n_1|^2 + |n_2|^2}^2} \nonumber \\
&\les C_\o N^{2 \al - 2 + \eps}
\label{scale}
\end{align}

\noi
for some $0 < C_\o < \infty$ almost surely and $\eps > 0$ arbitrarily small. In order to have $\| u_N^{(2)} \|_{H^{- \al} (\T^2)} \les 1$, we need $2 \al - 2 + \eps \leq 0$, which is equivalent to $\al < 1$.

The above computation shows that the probabilistic scaling critical regularity is $\al_* = 1$. Proposition \ref{PROP:ill}, however, shows that every non-zero Fourier coefficient of the Picard second iterate $z_N^{(2)} (t)$ diverges when $\al \geq \frac 34$ and $t \neq 0$, which happens before $\al$ reaches the critical value $\al_* = 1$. This shows that the probabilistic scaling introduced in \cite{DNY2} fails in the case of the quadratic nonlinearity $|u|^2$. We point out that this discrepancy is mainly due to the fact that the probabilistic scaling only considers the special case when all frequencies have comparable sizes, which oversimplifies the situation in the context of a quadratic nonlinearity. Also, this discrepancy is closely related to the fact that we are considering \textit{very} rough random initial data (rougher than the Gaussian free field initial data), which is in particular relevant in studying NLS with a polynomial nonlinearity of low degree and in low dimensions. See Remark \ref{RMK:ill_d} below. Similar phenomena also occur in the contexts of wave equations and stochastic parabolic equations. See Remark \ref{RMK:comp} and Remark \ref{RMK:noise} for further details.

We finish this subsection by stating several remarks.

\begin{remark} \rm
\label{RMK:rand}
We would like to point out that there is a gap ($\frac 12 \leq \al < \frac 34$) between our almost sure local well-posedness and probabilistic ill-posedness of the quadratic NLS \eqref{qNLS}. We would like to address this issue in a forthcoming work.

If some well-posedness results of the quadratic NLS \eqref{qNLS} can be achieved in the range $\frac 12 \leq \al < \frac 34$, this will imply that NLS behaves better than the nonlinear wave equation (NLW) in the quadratic case, which will be interesting because usually NLW behaves at least as well as NLS. See Remark \ref{RMK:comp} below or \cite{OO} for well-posedness issues of NLW with a quadratic nonlinearity. 
\end{remark}

\begin{remark} \rm
\label{RMK:ill_d}
The proof of Proposition \ref{PROP:ill}, the probabilistic ill-poseness result of the quadratic NLS \eqref{qNLS}, can easily be adapted to general dimensions. Specifically, on $\T^d$ for $d \in \N$, when $\al \geq \frac 54 - \frac{d}{4}$ and $n \neq 0$, any subsequence of $\{ \F_x z_N^{(2)} (t, n) \}_{N \in \N}$ is not tight. 

The probabilistic scaling for the quadratic NLS with nonlinearity $|u|^2$ can also be easily computed on general $\T^d$, on which the probabilistic scaling critical regularity is $\al_* = 2 - \frac{d}{2}$. We note that when $d = 1, 2$, every non-zero Fourier coefficient of the Picard second iterate diverges before $\al$ reaches the critical value $\al_*$.
\end{remark}

\begin{remark} \rm
We can also address ill-posedness issues of the quadratic NLS with nonlinearity $u^2$ or $\cj{u}^2$ with random initial data \eqref{ini} using a similar computation as in the case of $|u|^2$. Specifically, on $\T^d$, with either nonlinearity $u^2$ or nonlinearity $\cj{u}^2$, every Fourier coefficient of the Picard second iterate diverges (in the same sense of that in Proposition \ref{PROP:ill}) when $\al \geq 2 - \frac{d}{4}$. The reason for this different range of $\al$ from that in the context of nonlinearity $|u|^2$ is mainly due to the different phase functions corresponding to these nonlinearities (i.e. $-2 n \cdot n_2$ for $|u|^2$, $-2 n_1 \cdot n_2$ for $u^2$, and $|n|^2 + |n_1|^2 + |n_2|^2$ for $\cj{u}^2$).

We can also use a similar procedure as in \eqref{scale} to compute the probabilistic scaling for the quadratic NLS with nonlinearity $u^2$ or $\cj{u}^2$, each of which has the same critical regularity $\al_* = 2 - \frac{d}{2}$. It is interesting to note that in the context of nonlinearity $u^2$ or $\cj{u}^2$, the non-convergence of the Picard second iterate does not happen before $\al$ reaches the critical regularity.
\end{remark}

\begin{remark} \rm
\label{RMK:comp}
In \cite{OO}, Oh and Okamoto studied well-posedness issues of the stochastic nonlinear wave equation (NLW) with a quadratic nonlinearity on $\T^2$. Let us compare the situations for the quadratic NLS \eqref{qNLS} and the following quadratic NLW on $\T^2$:
\begin{align}
\begin{cases}
\dt^2 u + (1 - \Dl) u = u^2 \\
(u, \dt u)|_{t = 0} = (u_0^\o, u_1^\o),
\end{cases}
\label{NLW}
\end{align}

\noi
where
\begin{align*}
(u_0^\o, u_1^\o) = \bigg( \sum_{n \in \Z^2} \frac{ g_{0, n} (\o) }{\jb{n}} e^{in \cdot x}, \sum_{n \in \Z^2} \jb{n}^\al g_{1, n} (\o)  e^{in \cdot x} \bigg).
\end{align*}

\noi
Here, $\al \in \R$ and $\{ g_{0, n}, g_{1, n} \}_{n \in \Z^2}$ is a sequence of independent standard complex Gaussian random variables conditioned that $g_{j, -n} = \cj{g_{j, n}}$, for all $n \in \Z^2$, $j = 0, 1$. We point out that the probabilistic well-posedness and ill-posedness results in \cite{OO} for the quadratic SNLW also apply to \eqref{NLW} (with the standard Wick renormalization): \eqref{NLW} is almost surely locally well-posed when $\al < \frac 12$ and is probabilistically ill-posed in the sense that every Fourier coefficient of the Picard second iterate diverges almost surely when $\al \geq \frac 12$. 

We note that both the quadratic NLS \eqref{qNLS} and the quadratic NLW \eqref{NLW} are almost surely locally well-posed when $\al < \frac 12$. Regarding the probabilistic ill-posedness, for the quadratic NLS \eqref{qNLS}, every non-zero frequency diverges when $\al \geq \frac 34$; whereas for the quadratic NLW \eqref{NLW}, every frequency of the Picard second iterate diverges when $\al \geq \frac 12$, which also happens before reaching the critical regularity $\al_* = 1$ of \eqref{NLW}. See Proposition 1.6 in \cite{OO} for more details. The difference of the pathological behaviors of the two equations is mainly due to the different structures of the corresponding Duhamel operators.
\end{remark}

\begin{remark} \rm
\label{RMK:noise}
Let us also mention some failures of scaling analysis that happen in the context of parabolic equations forced by rough noises. In the past decade, there has been a huge progress in the study of stochastically forced parabolic equations using the theory of regularity structures introduced by Hairer \cite{Hair13, Hair14, Hair14s, Hair15}. In particular, the theory of regularity structures is able to solve a wide range of parabolic equations with a space-time white noise forcing that are subcritical according to the notion of local subcriticality introduced by Hairer \cite{Hair14}. However, when the stochastic forcing is rougher than the space-time white noise, the scaling analysis may fail to provide a prediction for well-posedness issues. For example, in \cite{Hosh}, Hoshino showed that for the KPZ equation driven by a fractional derivative of a space-time white noise, the standard solution theory breaks down before reaching the critical regularity. See also \cite{OO} for a similar phenomenon that occurs in the context of the stochastic nonlinear heat equation forced by a fractional derivative of a space-time white noise.
\end{remark}

\subsection{Organization of the paper}

This paper is organized as follows. In Section \ref{SEC:lem}, we introduce some notations, definitions, and preliminary lemmas. In Section \ref{SEC:bilin}, we establish bilinear estimates that are crucial for proving our almost sure local well-posedness result of the quadratic NLS \eqref{qNLS}. In Section \ref{SEC:well}, we prove Theorem \ref{THM:LWP}, the almost sure local well-posedness result of \eqref{qNLS} when $\al < \frac 12$. In Section \ref{SEC:ill}, we prove Proposition \ref{PROP:ill}, the probabilistic ill-posedness result of \eqref{qNLS} for $\al \geq \frac 34$.

\section{Notations and preliminary lemmas}
\label{SEC:lem}

In this section, we discuss some relevant notations and lemmas.

\subsection{Notations}

For a space-time distribution $u$ defined on $\R \times \T^2$, we write $\F_x u$ to denote the space Fourier transform of $u$ and we write $\ft u$ to denote the space-time Fourier transform of $u$. We also define the following twisted space-time Fourier transform:
\[ \wt u (\tau, k) = \ft u (\tau - |k|^2, k). \]

Given a dyadic number $N \in 2^{\Z_{\geq 0}}$, we let $P_N$ be the frequency projector onto the spatial frequencies $\{ n \in \Z^2: \frac{N}{2} < \jb{n} \leq N \}$, where $\jb{\cdot} = (1 + |\cdot|^2)^{\frac 12}$. For any subset $Q \subset \Z^2$, we let $P_Q$ be the frequency projector onto $Q$. Also, we use $P_{\neq 0}$ to denote the restriction to non-zero frequencies.


Let $\chi$ be a smooth cut-off function such that $\chi \equiv 1$ on $[-1, 1]$ and $\chi \equiv 0$ outside of $[-2, 2]$.

We use $A \les B$ to denote $A \leq CB$ for some constant $C > 0$, and we write $A \sim B$ if $A \les B$ and $B \les A$. Also, we write $A \ll B$ if $A \leq cB$ for some sufficiently small $c > 0$. In addition, we use $a+$ and $a-$ to denote $a + \eps$ and $a - \eps$, respectively, for sufficiently small $\eps > 0$.

\subsection{Fourier restriction norm method}
\label{SUBSEC:Xsb}

In this subsection, we introduce definitions and lemmas of $X^{s,b}$-spaces, also called the Bourgain spaces, due to Klainerman-Machedon \cite{KM93} and Bourgain \cite{Bour93}. Given $s, b \in \R$, we define the $X^{s, b} = X^{s,b}(\R \times \T^2)$ norm as
\begin{align*}
\| u \|_{X^{s,b}} := \big\| \jb{n}^s \jb{ \tau + |n|^2 }^b \ft u (\tau, n) \big\|_{L_\tau^2 \l_n^2 (\R \times \Z^2)}.
\end{align*} 

\noi
The space $X^{s,b}$ is then defined by the completion of functions that are $C^\infty$ in space and Schwartz in time with respect to this norm. For $T > 0$, we define the space $X_T^{s, b}$ by the restriction of distributions in $X^{s,b}$ onto the time interval $[-T, T]$ via the norm
\[ \| u \|_{X_T^{s,b}} := \inf \big\{ \| v \|_{X^{s, b}(\R \times \T^2)} : v|_{[-T, T]} = u \big\}.  \]

\noi
For any $s \in \R$ and $b > \frac 12$, we have $X_T^{s, b} \subset C([-T, T]; H^s(\T^2))$, where $H^s (\T^2)$ is the $L^2$-based Sobolev space on $\T^2$ with regularity $s$.

We define the truncated Duhamel operator as
\begin{align}
\I_\chi F(t) = \chi(t) \int_0^t \chi(t') e^{i(t - t') \Dl} F(t') \, dt'.
\label{defI}
\end{align}
We first recall the following linear estimates. See \cite{Bour93, GTV, Tao}.
\begin{lemma}
\label{LEM:lin}
Let $s \in \R$ and $b > \frac 12$. Then, we have
\[ \| \I_\chi F \|_{X^{s, b}} \les_{b}  \| F \|_{X^{s, b - 1}}. \]
\end{lemma}

Next, we recall the following $L^4$-Strichartz estimate. See \cite{Bour93, Bour95}.
\begin{lemma}
\label{LEM:L4}
Let $Q$ be a spatial frequency ball of radius $N$ (not necessarily centered at the origin). Then, we have
\[ \| P_Q u \|_{L_{t,x}^4 ([-1,1] \times \T^2)} \les N^{0+} \| u \|_{X^{0, \frac 12 -}}. \]
\end{lemma}

We also recall the following time localization estimate. For a proof, see Proposition 2.7 in \cite{DNY2}.
\begin{lemma}
\label{LEM:lin_t}
Let $\varphi$ be a Schwartz function, and let $\varphi_T (t) = \varphi(t/T)$ for $0 < T \leq 1$. Let $s \in \R$ and $\frac 12 < b \leq b_1 < 1$. Then, for any space-time function $u$ that satisfies $u(0, x) = 0$ for all $x \in \T^2$, we have
\[ \| \varphi_T \cdot u \|_{X^{s, b}} \les T^{b_1 - b} \| u \|_{X^{s, b_1}}. \]
\end{lemma}

Finally, we record the following lemma. For a proof, see Lemma 3.1 in \cite{DNY1}.
\begin{lemma}
\label{LEM:Duh}
For all $\tau \in \R$ and $n \in \Z^2$, we have the formula
\[ \wt{\I_\chi F} (\tau, n) = \int_\R  K (\tau, \tau') \wt F (\tau', n) \, d\tau', \]

\noi
where the kernel $K$ satisfies
\[ |K (\tau, \tau')| \les \bigg( \frac{1}{\jb{\tau}^3} + \frac{1}{\jb{\tau - \tau'}^3} \bigg) \frac{1}{\jb{\tau'}} \les \frac{1}{\jb{\tau} \jb{\tau - \tau'}}. \]
\end{lemma}

\subsection{Counting estimates and a convolution lemma}

In this subsection, we recall some counting estimates and a convolution lemma. We first record the following fact from number theory. For a proof, see Lemma 4.3 in \cite{DNY2}.
\begin{lemma}
\label{LEM:nt}
Let $a_0, b_0 \in \C$. Let $m \in \Z [i]$ be such that $m \neq 0$. Let $M_1, M_2 > 0$. Then, the number of tuples $(a, b) \in (\Z [i])^2$ that satisfies
\[ ab = m, |a - a_0| \leq M_1, |b - b_0| \leq M_2 \]
is $O(M_1^\eps M_2^\eps)$ for any small $\eps > 0$, where the constant depends only on $\eps$.
\end{lemma}

We now show the following counting estimates.
\begin{lemma}
\label{LEM:count1}
Let $N, N_1, N_2 \geq 1$ be dyadic numbers. Let $n, n_1, n_2 \in \Z^2$ be such that $n$ lies in a ball of radius $N$, $n_1$ lies in a ball of radius $N_1$, $n_2$ lies in a ball of radius $N_2$, $n - n_1 + n_2 = 0$, and $|n|^2 - |n_1|^2 + |n_2|^2 = m$ for some fixed $m \in \Z$.

\smallskip \noi
\textup{(i)} The number of tuples $(n, n_1, n_2) \in (\Z^2)^3$ that satisfy the above conditions is $O(N_1 N_2 \max\{ N_1^\eps, N_2^\eps \})$ for any small $\eps > 0$, where the constant depends only on $\eps$.

\smallskip \noi
\textup{(ii)} If $n_1$ is fixed, then the number of tuples $(n, n_2) \in (\Z^2)^2$ that satisfy the above conditions is $O(\max\{N^\eps, N_2^\eps\})$ for any small $\eps > 0$, where the constant depends only on $\eps$.

\smallskip \noi
\textup{(iii)} If $n_2$ is fixed and $n_2 \neq 0$, then the number of tuples $(n, n_1) \in (\Z^2)^2$ that satisfy the above conditions is $O(\min\{ N, N_1 \})$. 

\smallskip \noi
\textup{(iv)} If $n$ is fixed and $n \neq 0$, then the number of tuples $(n_1, n_2) \in (\Z^2)^2$ that satisfy the above conditions is $O(\min\{ N_1, N_2 \})$.
\end{lemma}

\begin{proof}
(i) See Lemma 4.3 in \cite{DNY2} for the proof of this part.

\medskip \noi
(ii) Since $n_1$ is fixed, we know that $n + n_2 = n_1$ is fixed. Let $k = (k_1, k_2) = n - n_2$, so that we have that
\[ (k_1 + ik_2)(k_1 - ik_2) = |k|^2 = 2|n|^2 + 2|n_2|^2 - |n + n_2|^2 = 2m + |n_1|^2 \]
is fixed. Since $k = n - n_2$ lies in a ball of radius $\leq N + N_2$, by Lemma \ref{LEM:nt}, we know that the number of choices for $k$ is $O(\max\{N^\eps, N_2^\eps\})$ for any small $\eps > 0$. Thus, the number of choices for $(n, n_2)$ is $O(\max\{N^\eps, N_2^\eps\})$ for any small $\eps > 0$.

\medskip \noi
(iii) Note that since $n = n_1 - n_2$, we have
\[ m = |n_1 - n_2|^2 - |n_1|^2 + |n_2|^2 = -2 n_1 \cdot n_2 + 2 |n_2|^2. \]
This shows that $n_1 \cdot n_2$ is fixed, which means that $n_1$ is restricted to a line. Also, we have
\[ m = |n|^2 - |n + n_2|^2 + |n_2|^2 = -2 n \cdot n_2. \]
This shows that $n \cdot n_2$ is fixed, which means that $n$ is restricted to a line. 
Thus, the number of choices for $(n, n_1)$ is $O(\min \{ N, N_1 \})$. 

\medskip \noi
(iv) The proof of this part is the same as that in part (iii). Thus, we omit details.
\end{proof}

We end this subsection by recording the following convolution inequality. For a proof, see Lemma 4.2 in \cite{GTV}.
\begin{lemma}
\label{LEM:conv}
Let $0 \leq \beta \leq \g$ with $\g > 1$. Then, for any $a \in \R$, we have
\[ \int_\R \frac{1}{\jb{x}^\beta \jb{x - a}^\g}  \les \frac{1}{\jb{a}^\beta}. \]
\end{lemma}

\subsection{Tools from stochastic analysis}
\label{SUBSEC:prob}

In this subsection, we recall the Wiener chaos estimate. Let $(H, B, \mu)$ be an abstract Wiener space, where $\mu$ is a Gaussian measure on a separable Banach space $B$ and $H \subset B$ is its Cameron-Martin space. Let $\{ e_j \}_{j \in \N} \subset B$ be an orthonormal system of $H^* = H$. We define a polynomial chaos of order $k$ as an element of the form $\prod_{j = 1}^\infty H_{k_j} ( \langle x, e_j \rangle )$. Here, $x \in B$, $k_j \neq 0$ for finitely many $j$'s, $k = \sum_{j = 1}^\infty k_j$, $H_{k_j}$ is the Hermite polynomial of degree $k_j$, and $\langle \cdot, \cdot \rangle = \vphantom{|}_B \langle \cdot, \cdot \rangle_{B^*}$ denotes the $B - B^*$ duality pairing. We denote the closure of all polynomial chaoses of order $k$ under $L^2 (B, \mu)$ by $\H_k$, whose elements are called homogeneous Wiener chaoses of order $k$. We also denote
\begin{align}
\H_{\leq k} = \bigoplus_{j = 0}^k \H_j
\label{Hk}
\end{align}
for $k \in \N$.

Let $L$ be the Ornstein-Uhlenbeck operator. It is known that any element in $\H_k$ is an eigenfunction of $L$ with eigenvalue $-k$. Then, we have the following Wiener chaos estimate \cite[Theorem I.22]{Sim} as a consequence of the hypercontractivity of the Ornstein-Uhlenbeck semigroup $U(t) = e^{tL}$ due to Nelson \cite{Nel}.

\begin{lemma}
\label{LEM:wie}
Let $k \in \N$. Then, for any $p \geq 2$ and $X \in \H_{\leq k}$, we have
\begin{align*}
\E \big[ |X|^p \big]^{\frac 1p} \leq (p - 1)^{\frac{k}{2}} \E \big[ |X|^2 \big]^{\frac 12}.
\end{align*}
\end{lemma}

\subsection{Random tensor and deterministic tensor estimates}
\label{SUBSEC:rand}
In this subsection, we recall some useful results of random tensor estimates developed in \cite{DNY3} and also prove some deterministic tensor estimates.

Let us first recall the definition of (random) tensors. Let $A$ be a finite index set. We denote $n_A$ as the tuple $(n_j: j \in A)$. A \textit{tensor} $h = h_{n_A}$ is a function from $(\Z^2)^A$ to $\C$ with $n_A$ being the input variables. The \textit{support} of $h$ is the set of $n_A$ such that $h_{n_A} \neq 0$. Note that $h$ may also depend on $\o \in \O$, in which case $h$ is called a \textit{random tensor}.

Given a finite index set $A$, we define the norm $\| \cdot \|_{n_A}$ by
\begin{align*}
\| h \|_{n_A} = \| h \|_{\l_{n_A}^2} = \bigg( \sum_{n_A \in (\Z^2)^A} |h_{n_A}|^2 \bigg)^{1/2}.
\end{align*}

\noi
For any partition $(B, C)$ of $A$, i.e. $B \cup C = A$ and $B \cap C = \varnothing$, we define the norm $\| \cdot \|_{n_B \to n_C}$ by
\begin{align*}
\| h \|_{n_B \to n_C}^2 = \sup \bigg\{ \sum_{n_C \in (\Z^2)^C} \bigg| \sum_{n_B \in (\Z^2)^B} h_{n_A} \cdot f_{n_B} \bigg|^2 :  \sum_{n_B \in (\Z^2)^B} |f_{n_B}|^2 = 1 \bigg\}.
\end{align*}

\noi
For any tensor $h$, by duality, we have $\| h \|_{n_B \to n_C} = \| h \|_{n_C \to n_B} = \| \cj{h} \|_{n_B \to n_C}$. If either $B = \varnothing$ or $C = \varnothing$, we have $\| h \|_{n_B \to n_C} = \| h \|_{n_A}$.

We also need the following definitions to state the random tensor estimate. For a complex number $a$, we define $a^+ = a$ and $a^- = \cj{a}$. Let $A$ be a finite index set. For each $j \in A$, we associate $j$ with a sign $\zeta_j \in \{ \pm \}$. For $j_1, j_2 \in A$, we say that $(n_{j_1}, n_{j_2})$ is a \textit{pairing} if $n_{j_1} = n_{j_2}$ and $\zeta_{j_1} = - \zeta_{j_2}$. Also, recall that $\{ g_n \}_{n \in \Z^2}$ is a set of independent standard complex-valued Gaussian random variables. For each $n \in \Z^2$, we can write
\[ g_n (\o) = \rho_n (\o) \eta_n(\o), \]
where $\rho_n = |g_n|$ and $\eta_n = \rho_n^{-1} g_n$ are independent. Note that each $\eta_n$ is uniformly distributed on the unit circle of $\C$.

We now record the following random tensor estimate. For the proof, see Proposition 4.14 in \cite{DNY3}.
\begin{lemma}
\label{LEM:rten}
Let $0 < T \leq 1$. Let $h_{a_1 a_2 n_A} = h_{a_1 a_2 n_A} (\o)$ be a random tensor, where each $n_j \in \Z^2$ and $(a_1, a_2) \in (\Z^2)^q$ for some integer $q \geq 2$. Given a dyadic number $M \geq 1$, we assume that $\jb{a_1}, \jb{a_2} \les M$ and $\jb{n_j} \les M$ for all $j \in A$. We also assume that in the support of $h_{a_1 a_2 n_A}$, there is no pairing in $n_A$. Moreover, we assume that $\{ h_{a_1 a_2 n_A} \}$ is independent with $\{ \eta_n \}_{n \in \Z^2}$. Define the tensor 
\[ H_{a_1 a_2} = \sum_{n_A} h_{a_1 a_2 n_A} \prod_{j \in A} \eta_{n_j}^{\zeta_j}. \]

\noi
Then, there exists constants $C, c > 0$ such that outside an exceptional set of probability $\leq C \exp(- \frac{c M}{T^\ta})$ with $\ta > 0$, we have
\[ \| H_{a_1 a_2} \|_{a_1 \to a_2} \les T^{-\ta} M^\ta \cdot \max_{(A_1, A_2)} \| h \|_{a_1 n_{A_1} \to a_2 n_{A_2}}, \]

\noi
where $(A_1, A_2)$ runs over all partitions of $A$.
\end{lemma}

We also record the following variant of Lemma \ref{LEM:rten}. For the proof, see Proposition 4.15 in \cite{DNY3}.
\begin{lemma}
\label{LEM:rten2}
Consider the same setting as in Lemma \ref{LEM:rten} with the following differences:
\begin{itemize}
\item[(1)] We only restrict $\jb{n_j} \les M$ for all $j \in A$ but do not impose any condition on $\jb{a_1}$ or $\jb{a_2}$.

\item[(2)] We assume that $a_1, a_2 \in \Z^2$ and that in the support of the random tensor $h_{a_1 a_2 n_A}$ we have $|a_1 - \zeta a_2| \les M$ where $\zeta \in \{\pm\}$.

\item[(3)] The random tensor $h_{a_1 a_2 n_A}$ only depends on $a_1 - \zeta a_2$, $|a_1|^2 - \zeta |a_2|^2$, and $n_A$, and is supported in the set where $\big| |a_1|^2 - \zeta |a_2|^2 \big| \les M^{10}$.
\end{itemize}

\noi
Then, there exists constants $C, c > 0$ such that outside an exceptional set of probability $\leq C \exp(- \frac{c M}{T^\ta})$ with $\ta > 0$, we have
\[ \| H_{a_1 a_2} \|_{a_1 \to a_2} \les T^{-\ta} M^\ta \cdot \max_{(A_1, A_2)} \| h \|_{a_1 n_{A_1} \to a_2 n_{A_2}}, \]

\noi
where $(A_1, A_2)$ runs over all partitions of $A$.
\end{lemma}

We now turn our attention to some deterministic tensor estimates. Given $m \in \Z$, we define the base tensor $h^m_{n n_1 n_2}$ as
\begin{align}
\label{defh}
h^m_{n n_1 n_2} = \ind_{n - n_1 + n_2 = 0} \ind_{|n|^2 - |n_1|^2 + |n_2|^2 = m}.
\end{align}

\noi
We now show the following estimates regarding the base tensor $h^m_{n n_1 n_2}$.
\begin{lemma}
\label{LEM:dten}
Let $N, N_1, N_2 \geq 1$ be dyadic numbers and let $\eps > 0$ be arbitrarily close to 0. Let $J$ be a ball of radius $\sim N$, $J_1$ be a ball of radius $\sim N_1$, and $J_2$ be a ball of radius $\sim N_2$. We define
\[ S := \{ (n, n_1, n_2) \in (\Z^2)^3: n \in J, n_1 \in J_1, n_2 \in J_2 \}. \]

\noi
Thus, we have the following estimates:
\begin{align}
\| h^m_{n n_1 n_2} \cdot \ind_S \|_{n n_1 n_2} &\les N_1^{\frac 12} N_2^{\frac 12} \max \{ N_1^\eps, N_2^\eps \}, \label{012} \\
\| h^m_{n n_1 n_2} \cdot \ind_S \|_{n_1 \to n n_2} &\les \max \{ N^\eps, N_2^\eps \}, \label{1-02} \\
\| h^m_{n n_1 n_2} \cdot \ind_S \cdot \ind_{n_2 \neq 0} \|_{n_2 \to n n_1} &\les \min \big\{ N^{\frac 12}, N_1^{\frac 12} \big\}, \label{2-01} \\
\| h^m_{n n_1 n_2} \cdot \ind_S \cdot \ind_{n \neq 0} \|_{n \to n_1 n_2} &\les \min \big\{ N_1^{\frac 12}, N_2^{\frac 12} \big\}. \label{0-12}
\end{align}
\end{lemma}

\begin{proof}
For \eqref{012}, we use Lemma \ref{LEM:count1} (i) to obtain
\[ \| h^m_{n n_1 n_2} \cdot \ind_S \|_{n n_1 n_2} \les N_1^{\frac 12} N_2^{\frac 12} \max \{ N_1^\eps, N_2^\eps \}. \]

\noi
For \eqref{1-02}, we use Schur's test and Lemma \ref{LEM:count1} (ii) to obtain
\begin{align*}
\| h^m_{n n_1 n_2} \cdot \ind_S \|_{n_1 \to n n_2} &\leq \bigg( \sup_{n_1} \sum_{n, n_2}  h^m_{n n_1 n_2} \cdot \ind_S  \bigg)^{1/2}  \bigg( \sup_{n, n_2} \sum_{n_1}  h^m_{n n_1 n_2} \cdot \ind_S  \bigg)^{1/2}  \\
&\les \max \{ N^\eps, N_2^\eps \}.
\end{align*}

\noi
For \eqref{2-01}, we use Schur's test and Lemma \ref{LEM:count1} (iii) to obtain
\begin{align*}
\| h^m_{n n_1 n_2} \cdot \ind_S \cdot  \ind_{n_2 \neq 0} \|_{n_2 \to n n_1} &\leq \bigg( \sup_{n_2 \neq 0} \sum_{n, n_1}  h^m_{n n_1 n_2} \cdot \ind_S  \bigg)^{1/2}  \bigg( \sup_{n, n_1} \sum_{n_2}  h^m_{n n_1 n_2} \cdot \ind_S  \bigg)^{1/2}  \\
&\les \min \big\{ N^{\frac 12}, N_1^{\frac 12} \big\}.
\end{align*}

\noi
For \eqref{0-12}, we use Schur's test and Lemma \ref{LEM:count1} (iv) to obtain
\begin{align*}
\| h^m_{n n_1 n_2} \cdot \ind_S \cdot \ind_{n \neq 0} \|_{n \to n_1 n_2} &\leq \bigg( \sup_{n \neq 0} \sum_{n_1, n_2}  h^m_{n n_1 n_2} \cdot \ind_S  \bigg)^{1/2}  \bigg( \sup_{n_1, n_2} \sum_{n}  h^m_{n n_1 n_2} \cdot \ind_S  \bigg)^{1/2}  \\
&\les \min \big\{ N_1^{\frac 12}, N_2^{\frac 12} \big\}.
\end{align*}

\noi
We thus finish our proof.
\end{proof}

\begin{remark} \rm
\label{RMK:ten}
The condition $n_2 \neq 0$ in the estimate \eqref{2-01} is necessary in view of the restriction $n_2 \neq 0$ in Lemma \ref{LEM:count1} (iii). Similarly, the condition $n \neq 0$ in the estimate \eqref{0-12} is necessary in view of the restriction $n \neq 0$ in Lemma \ref{LEM:count1} (iv).
\end{remark}

\section{Bilinear estimates}
\label{SEC:bilin}

In this section, we establish several bilinear estimates that are crucial for proving Theorem \ref{THM:LWP}, the almost sure local well-posedness result of the quadratic NLS \eqref{qNLS}. Specifically, we need to estimate the following term
\[ \big\| \varphi_T \cdot \I_\chi \big( v^{(1)} \cj{v^{(2)}} \big) \big\|_{X^{s, \frac 12 + \dl}}, \]

\noi
where $s, \dl > 0$ are sufficiently small, $\varphi_T(t) = \varphi(t/T)$ with $\varphi$ being a Schwartz function and $0 < T \leq 1$, and $\I_\chi$ is the truncated Duhamel operator as defined in \eqref{defI} with $\chi$ being a smooth cut-off function such that $\chi \equiv 1$ on $[-1, 1]$ and $\chi \equiv 0$ outside of $[-2, 2]$. Here, each of $v^{(1)}$ and $v^{(2)}$ is either an arbitrary space-time function on $\R \times \T^2 $ or the random linear solution with a time cut-off $\chi \cdot z$, where $z$ is as defined in \eqref{defz}.

We first consider the case when neither $v^{(1)}$ nor $v^{(2)}$ is $\chi \cdot z$. Specifically, we show the following bilinear estimate.
\begin{proposition}
\label{PROP:vv}
Let $s > 0$ and let $\dl > 0$ be sufficiently small. Let $0 < T \leq 1$. Then, we have
\[ \big\| \varphi_T \cdot \I_\chi \big( v^{(1)} \cj{v^{(2)}} \big) \big\|_{X^{s, \frac 12 + \dl}} \les T^{\dl} \| v^{(1)} \|_{X^{s, \frac 12 + \dl}} \| v^{(2)} \|_{X^{s, \frac 12 + \dl}}. \]
\end{proposition}

\begin{proof}
By Lemma \ref{LEM:lin_t} and Lemma \ref{LEM:lin}, we have
\begin{align}
\big\| \varphi_T \cdot \I_\chi \big( v^{(1)} \cj{v^{(2)}} \big) \big\|_{X^{s, \frac 12 + \dl}} \les T^\dl \big\| \I_\chi \big( v^{(1)} \cj{v^{(2)}} \big) \big\|_{X^{s, \frac 12 + 2\dl}} \les T^\dl \big\| v^{(1)} \cj{v^{(2)}} \big\|_{X^{s, - \frac 12 + 2 \dl}}. 
\label{vv1}
\end{align}

\noi
By duality and dyadic decomposition, we have
\begin{align}
\big\| v^{(1)} \cj{v^{(2)}} \big\|_{X^{s, - \frac 12 + 2 \dl}} &= \sup_{\| w \|_{X^{0, \frac 12 - 2\dl}} \leq 1} \bigg| \int_\R \int_{\T^2} \jb{\nabla}^s \big( v^{(1)} \cj{v^{(2)}} \big) \cj{w} \,dxdt  \bigg| \nonumber \\
&\les \sup_{\| w \|_{X^{0, \frac 12 - 2\dl}} \leq 1} \sum_{\substack{N, N_1, N_2 \geq 1 \\ \text{dyadic}}} \bigg| \int_\R \int_{\T^2} \jb{\nabla}^s \big( P_{N_1} v^{(1)} \cj{ P_{N_2} v^{(2)} } \big) \cj{P_N w} \,dxdt  \bigg|.
\label{vv2}
\end{align}

\noi
Let $n_1, n_2, n$ be the frequencies corresponding to the three terms $P_{N_1} v^{(1)}, P_{N_2} v^{(2)}, P_N w$, respectively. In order for the above integral on $\T^2$ to be non-zero, we must have $n_1 - n_2 - n = 0$. This leads us to the following three cases.

\medskip \noi
\textbf{Case 1:} $N_1 \sim N_2$.

\noi
In this case, we have $N \les N_1 \sim N_2$. By H\"older's inequality and Lemma \ref{LEM:L4}, we have
\begin{align}
\bigg| \int_\R \int_{\T^2} &\jb{\nabla}^s \big( P_{N_1} v^{(1)} \cj{ P_{N_2} v^{(2)} } \big) \cj{P_N w} \,dxdt  \bigg| \nonumber \\
&\les N^s \| P_{N_1} v^{(1)} \|_{L_{t,x}^4} \| P_{N_2} v^{(2)} \|_{L_{t,x}^4} \| P_N w \|_{L_{t,x}^2} \nonumber \\
&\les N_1^{0-} \big\| N_1^{\frac{s}{2} +} P_{N_1} v^{(1)} \big\|_{L_{t,x}^4} \big\| N_2^{\frac{s}{2}} P_{N_2} v^{(2)} \big\|_{L_{t,x}^4} \| P_N w \|_{L_{t,x}^2} \nonumber \\
&\les N_1^{0-} \| P_{N_1} v^{(1)} \|_{X^{s, \frac 12 + \dl}} \| P_{N_2} v^{(2)} \|_{X^{s, \frac 12 + \dl}} \| P_N w \|_{X^{0, 0}} \nonumber \\
&\leq N_1^{0-} \| v^{(1)} \|_{X^{s, \frac 12 + \dl}} \| v^{(2)} \|_{X^{s, \frac 12 + \dl}} \| w \|_{X^{0, \frac 12 - 2\dl}}.
\label{vv3}
\end{align}

\noi
Combining \eqref{vv1}, \eqref{vv2}, \eqref{vv3} and summing over $N_1 \sim N_2 \ges N$, we obtain the desired estimate.

\medskip \noi
\textbf{Case 2:} $N_1 \gg N_2$.

\noi
In this case, we have $N \sim N_1 \gg N_2$. We partition the annulus $\{ |n_1| \sim N_1 \}$ into balls of radius $\sim N_2$ and denote the set of these balls as $\mathcal{J}_1$, and we partition the annulus $\{ |n| \sim N \}$ into balls of radius $\sim N_2$ and denote the set of these balls as $\mathcal{J}$. Note that for each fixed $J_1 \in \mathcal{J}_1$, the product $\ind_{J_1} (n_1) \cdot \ind_J (n)$ is non-zero for at most a fixed constant number of $J \in \mathcal{J}$, and we denote the set of these $J$'s as $\mathcal{J}(J_1)$. Thus, by H\"older's inequality, Lemma \ref{LEM:L4}, and the Cauchy-Schwarz inequality in $J_1$, we have
\begin{align}
\bigg| \int_\R &\int_{\T^2} \jb{\nabla}^s \big( P_{N_1} v^{(1)} \cj{ P_{N_2} v^{(2)} } \big) \cj{P_N w} \,dxdt  \bigg| \nonumber \\
&\les \sum_{J_1 \in \mathcal{J}_1} \sum_{J \in \mathcal{J}(J_1)} N_1^s \| P_{J_1} P_{N_1} v^{(1)} \|_{L_{t,x}^4} \| P_{N_2} v^{(2)} \|_{L_{t,x}^4} \| P_J P_N w \|_{L_{t,x}^2} \nonumber \\
&\les \sum_{J_1 \in \mathcal{J}_1} \sum_{J \in \mathcal{J}(J_1)} N_1^s N_2^{0+} \| P_{J_1} P_{N_1} v^{(1)} \|_{X^{0, \frac 12 + \dl}} \| P_{N_2} v^{(2)} \|_{X^{0, \frac 12 + \dl}} \| P_J P_N w \|_{X^{0, \frac 12 - 2\dl}} \nonumber \\
&\les N_1^s N_2^{0-} \| P_{N_1} v^{(1)} \|_{X^{0, \frac 12 + \dl}} \| P_{N_2} v^{(2)} \|_{X^{s, \frac 12 + \dl}} \| P_N w \|_{X^{0, \frac 12 - 2 \dl}} \nonumber \\
&\sim N_2^{0-} \| P_{N_1} v^{(1)} \|_{X^{s, \frac 12 + \dl}} \| P_{N_2} v^{(2)} \|_{X^{s, \frac 12 + \dl}} \| P_N w \|_{X^{0, \frac 12 - 2 \dl}}.
\label{vv4}
\end{align}

\noi
Combining \eqref{vv1}, \eqref{vv2}, \eqref{vv4}, applying the Cauchy-Schwarz inequality in $N_1 \sim N$, and summing over $N_1 \sim N \gg N_2$, we obtain the desired estimate.

\medskip \noi
\textbf{Case 3:} $N_1 \ll N_2$.

\noi
The steps in this case are the same as those in Case 2 by switching the roles of $v^{(1)}$ and $v^{(2)}$, so that we omit details.
\end{proof}

We now consider the case when at least one of $v^{(1)}$ and $v^{(2)}$ is the random linear solution with a time cut-off $\chi \cdot z$. Our goal is to prove the following estimates. The idea of the computations in the proof comes from \cite{Wang}.
\begin{proposition}
\label{PROP:zz}
Let $\al < \frac 12$. Let $s, \dl > 0$ be sufficiently small. Let $0 < T \leq 1$.

\medskip \noi
\textup{(i)} We have
\begin{align}
\big\| P_{\neq 0} \big( \varphi_T \cdot \I_\chi \big( v \cdot \cj{\chi \cdot z} \big) \big) \big\|_{X^{s, \frac 12 + \dl}} \les T^{\dl - 2 \ta} \| v \|_{X^{s, \frac 12 + \dl}}
\label{vz}
\end{align}

\noi
outside an exceptional set of probability $\leq C \exp(- \frac{c}{T^\ta})$ with $C, c > 0$ being constants and $0 < \ta \ll \dl$.

\medskip \noi
\textup{(ii)} If $v$ has mean zero (i.e. has no zeroth frequency term), we have
\begin{align}
\big\| P_{\neq 0} \big( \varphi_T \cdot \I_\chi \big( \chi \cdot z \cdot \cj{v} \big) \big) \big\|_{X^{s, \frac 12 + \dl}} \les T^{\dl - 2 \ta} \| v \|_{X^{s, \frac 12 + \dl}}
\label{zv}
\end{align}

\noi
outside an exceptional set of probability $\leq C \exp(- \frac{c}{T^\ta})$ with $C, c > 0$ being constants and $0 < \ta \ll \dl$.

\medskip \noi
\textup{(iii)} We have
\begin{align}
\big\| P_{\neq 0} \big( \varphi_T \cdot \I_\chi \big( \chi \cdot z \cdot \cj{\chi \cdot z} \big) \big) \big\|_{X^{s, \frac 12 + \dl}} \les T^{\dl - 2 \ta}
\label{zz}
\end{align}

\noi
outside an exceptional set of probability $\leq C \exp(- \frac{c}{T^\ta})$ with $C, c > 0$ being constants and $0 < \ta \ll \dl$.
\end{proposition}

\begin{proof}
We first do the following general setup. Let $v^{(1)}$ and $v^{(2)}$ be two space-time functions. By Lemma \ref{LEM:lin_t}, Lemma \ref{LEM:Duh}, duality, and dyadic decomposition, we have
\begin{align}
&\big\| P_{\neq 0} \big( \varphi_T \cdot \I_\chi \big( v^{(1)} \cj{v^{(2)}} \big) \big) \big\|_{X^{s, \frac 12 + \dl}} \nonumber \\
&\les T^\dl \big\| P_{\neq 0} \big( \I_\chi \big( v^{(1)} \cj{v^{(2)}} \big) \big) \big\|_{X^{s, \frac 12 + 2\dl}} \nonumber \\
&= T^\dl \bigg\| \ind_{n \neq 0} \jb{n}^s \jb{\tau}^{\frac 12 + 2\dl} \int_\R K(\tau, \tau') \wt{v^{(1)} \cj{v^{(2)}}} (\tau', n) \, d\tau' \bigg\|_{\l_n^2 L_\tau^2} \nonumber \\
&= T^\dl \sup_{\| \wt w \|_{\l_n^2 L_\tau^2} \leq 1} \bigg| \sum_{\substack{n, n_1, n_2 \in \Z^2 \\ n_1 - n_2 = n \neq 0}} \jb{n}^s \int \int \int  K \big( \tau, |n|^2 + (\tau_1 - |n_1|^2) - (\tau_2 - |n_2|^2) \big) \nonumber \\
&\quad \times \jb{\tau}^{\frac 12 + 2\dl} \wt{v^{(1)}} (\tau_1, n_1) \cj{\wt{v^{(2)}}} (\tau_2, n_2) \cj{\wt w}(\tau, n) \, d\tau d\tau_1 d\tau_2 \bigg| \nonumber \\
&\les T^\dl \sup_{\| \wt w \|_{\l_n^2 L_\tau^2} \leq 1} \sum_{\substack{N, N_1, N_2 \geq 1 \\ \text{dyadic}}} \bigg| \sum_{\substack{n, n_1, n_2 \in \Z^2 \\ n_1 - n_2 = n \neq 0}} \jb{n}^s \nonumber \\
&\quad \times \int \int \int  K \big( \tau, |n|^2 + (\tau_1 - |n_1|^2) - (\tau_2 - |n_2|^2) \big) \jb{\tau}^{\frac 12 + 2\dl} \nonumber \\
&\quad \times \wt{P_{N_1} v^{(1)}} (\tau_1, n_1) \cj{\wt{P_{N_2} v^{(2)}}} (\tau_2, n_2) \cj{\wt{P_N w}}(\tau, n) \, d\tau d\tau_1 d\tau_2 \bigg|,
\label{setup1}
\end{align}

\noi
where the kernel $K$ satisfies
\begin{align}
|K(\tau, \tau')| \les \frac{1}{\jb{\tau} \jb{\tau - \tau'}}.
\label{K_bdd}
\end{align}

\noi
We now separately discuss the three situations (i), (ii), and (iii).

\medskip \noi
(i) We consider the following two cases.

\medskip \noi
\textbf{Case 1:} $\jb{\tau} \gg N_2^{10}$. 

In this case, by H\"older's inequality in $n_2$, \eqref{K_bdd}, the Cauchy-Schwarz inequalities in $\tau_1, \tau$, and $n$, and Lemma \ref{LEM:conv}, we have
\begin{align}
\eqref{setup1} &\les T^\dl \sup_{\| \wt w \|_{\l_n^2 L_\tau^2} \leq 1} \sum_{\substack{N, N_1, N_2 \geq 1 \\ \textup{dyadic}}} N^s N_2^{-5 + 20\dl} N_2^2 N_2^{-1 + \al} \nonumber \\
&\quad \times \sup_{\substack{n_2 \in \Z^2 \\ \jb{n_2} \sim N_2}} \sum_{\substack{n \in \Z^2 \\ \jb{n} \sim N}} \int \int \int \jb{ (\tau - |n|^2) - (\tau_1 - |n_1|^2) + (\tau_2 - |n_2|^2) }^{-1} \nonumber \\
&\quad \times \big| \wt{P_{N_1} v} (\tau_1, n + n_2) \big| |g_{n_2}(\o) \ft \chi (\tau_2)| \big| \cj{\wt{P_N w}}(\tau, n) \big| \, d\tau_1 d\tau_2 d\tau 
\nonumber \\
&\les T^\dl \sup_{\| \wt w \|_{\l_n^2 L_\tau^2} \leq 1} \sum_{\substack{N, N_1, N_2 \geq 1 \\ \textup{dyadic}}} N^s N_2^{-4 + 20\dl + \al} \sup_{\substack{ n_2 \in \Z^2 \\ \jb{n_2} \sim N_2 }} |g_{n_2} (\o)| \nonumber \\
&\quad \times \big\| \jb{\tau_1}^{\frac 12 + \dl} \wt{P_{N_1} v} (\tau_1, n_1) \big\|_{\l_{n_1}^2 L_{\tau_1}^2} \big\| \wt{P_N w} (\tau, n) \big\|_{\l_n^2 L_\tau^2} \nonumber  \\
&\les T^\dl \sup_{\| w \|_{L_{t, x}^2} \leq 1} \sum_{\substack{N, N_1, N_2 \geq 1 \\ \textup{dyadic}}} N^s N_1^{-s} N_2^{-4 + 20\dl + \al}  \sup_{\substack{ n_2 \in \Z^2 \\ \jb{n_2} \sim N_2 }} |g_{n_2} (\o)| \nonumber\\
&\quad \times \| P_{N_1} v \|_{X^{s, \frac 12 + \dl}} \| P_N w \|_{L_{t, x}^2}.
\label{obs}
\end{align}

\noi
Note that we have the following Gaussian tail bound:
\begin{align}
\sum_{\substack{ n_2 \in \Z^2 \\ \jb{n_2} \sim N_2}} P(|g_{n_2}| > T^{-\ta} N_2^\dl) < C \exp\Big(-c \frac{N_2^\dl}{T^{\ta}} \Big)
\label{gauss}
\end{align}

\noi
for some constants $C, c > 0$ and $0 < \ta \ll \dl$, so that \eqref{obs} gives
\begin{align}
\eqref{setup1} \les T^{\dl - \ta} \sup_{\| w \|_{L_{t, x}^2} \leq 1} \sum_{\substack{N, N_1, N_2 \geq 1 \\ \textup{dyadic}}} N^s N_1^{-s} N_2^{-4 + 21\dl + \al}  \| P_{N_1} v \|_{X^{s, \frac 12 + \dl}} \| P_N w \|_{L_{t, x}^2} 
\label{obs2}
\end{align}

\noi
outside an exceptional set of probability $\leq C \exp(-c N_2^\dl / T^\ta)$. Recall that $\delta$ and $s$ can be made sufficiently small and $\al < \frac 12$. If $N \gg N_1$, we have $N \sim N_2$, so that we can use $N^s \sim N^{0-} N_2^{s+}$ and sum up dyadic $N, N_1, N_2$ in \eqref{obs2} to obtain \eqref{vz}. If $N \ll N_1$, we have $N_1 \sim N_2$, so that we can use $N^s \ll N^{0-} N_2^{s+}$ and sum up dyadic $N, N_1, N_2$ in \eqref{obs2} to obtain \eqref{vz}. If $N \sim N_1$, we can use the Cauchy-Schwarz inequality in $N \sim N_1$ and sum up dyadic $N, N_1, N_2 \geq 1$ in \eqref{obs2} to obtain \eqref{vz}.

\medskip \noi
\textbf{Case 2:} $\jb{\tau} \les N_2^{10}$.

We further split this case into two subcases.

\medskip \noi
\textbf{Subcase 2.1:} $N \les N_2$.

In this case, we have $N_1 \les N_2$. By the Cauchy-Schwarz inequalities in $\tau$ and $n$, \eqref{K_bdd}, and Minkowski's inequality, we have
\begin{align}
\eqref{setup1} &\les T^\dl \sup_{\| \wt w \|_{\l_n^2 L_\tau^2} \leq 1} \sum_{\substack{N, N_1, N_2 \geq 1 \\ \textup{dyadic}}} N^s  N_2^{30\dl} \nonumber \\
&\quad \times \bigg| \sum_{\substack{n, n_1, n_2 \in \Z^2 \\ n_1 - n_2 = n \neq 0}} \int \int \int  K \big( \tau, |n|^2 + (\tau_1 - |n_1|^2) - (\tau_2 - |n_2|^2) \big)  \jb{\tau}^{\frac 12 - \dl} \nonumber \\
&\quad \times \wt{P_{N_1} v} (\tau_1, n_1) \frac{\cj{g_{n_2}} (\o)}{\jb{n_2}^{1 - \al}} \cj{\ft \chi} (\tau_2) \cj{\wt{P_N w}}(\tau, n) \, d\tau d\tau_1 d\tau_2 \bigg| \nonumber \\
&\les T^\dl \sum_{\substack{N, N_1, N_2 \geq 1 \\ \textup{dyadic}}} N^s  N_2^{30\dl} \bigg[ \int \jb{\tau}^{-1 - 2\dl} \nonumber \\
&\quad \times \bigg( \sum_{m \in \Z} \int \int \jb{\tau - \tau_1 + \tau_2 - m}^{-1} \jb{\tau_1}^{- \frac 12 - \dl} \cj{\ft \chi} (\tau_2) \nonumber \\
&\quad \times \bigg\| \sum_{n_1, n_2 \in \Z^2} h^m_{n n_1 n_2} \ind_{S_1} \cdot \frac{\cj{g_{n_2}} (\o)}{\jb{n_2}^{1 - \al}}  \jb{\tau_1}^{\frac 12 + \dl} \wt{P_{N_1} v} (\tau_1, n_1) \bigg\|_{\l_n^2} d\tau_1 d\tau_2 \bigg)^2 d\tau \bigg]^{1/2},
\label{vz1}
\end{align}

\noi
where $h^m_{n n_1 n_2}$ is the base tensor as defined in \eqref{defh} and $S_1$ is a set defined by 
\begin{align}
S_1 : \! &= S_1 (N, N_1, N_2) \nonumber \\
&= \{ (n, n_1, n_2) \in (\Z^2)^3: n \neq 0, |n| \sim N, |n_1| \sim N_1, |n_2| \sim N_2 \}.
\label{s1}
\end{align}

\noi
Note that for $(n, n_1, n_2)$ restricted in $S_1$, we have $\les N_2^2$ choices for the value
\[ m = |n|^2 - |n_1|^2 + |n_2|^2, \]

\noi
which implies that
\begin{align}
\sum_{m \in \Z} \jb{\tau - \tau_1 + \tau_2 - m}^{-1} \les \log (1 + N_2^2) \les N_2^\dl.
\label{summ}
\end{align}

\noi
Thus, continuing with \eqref{vz1}, by H\"older's inequality in $m$, \eqref{summ}, the Cauchy-Schwarz inequality in $\tau_1$, we obtain
\begin{align}
\eqref{setup1} \les T^\dl \sum_{\substack{N, N_1, N_2 \geq 1 \\ \textup{dyadic}}} N^s  N_2^{31\dl} \sup_{m \in \Z} \bigg\| \sum_{n_2 \in \Z^2} h^m_{n n_1 n_2} \ind_{S_1} \cdot \frac{\cj{g_{n_2}} (\o)}{\jb{n_2}^{1 - \al}} \bigg\|_{n \to n_1} \| P_{N_1} v \|_{X^{s, \frac 12 + \dl}}.
\label{vz2}
\end{align}

\noi
By Lemma \ref{LEM:rten}, the Gaussian tail bound \eqref{gauss}, and Lemma \ref{LEM:dten}, we have
\begin{align}
\bigg\| &\sum_{n_2 \in \Z^2} h^m_{n n_1 n_2} \ind_{S_1} \cdot \frac{\cj{g_{n_2}} (\o)}{\jb{n_2}^{1 - \al}} \bigg\|_{n \to n_1} \nonumber \\
&\les T^{-2 \ta} N_2^{-1 + 2\dl + \al} \max \big\{ \| h^m_{n n_1 n_2} \ind_{S_1} \|_{n n_2 \to n_1}, \| h^m_{n n_1 n_2} \ind_{S_1} \|_{n \to n_1 n_2} \big\} \nonumber \\
&\les T^{-2 \ta} N_2^{-\frac 12 + 2\dl + \al}
\label{ten}
\end{align}

\noi
outside an exceptional set of probability $\leq C \exp(-c N_2^\dl / T^\ta)$ for some universal constants $C, c > 0$. Thus, combining \eqref{vz2} and \eqref{ten}, using the fact that $\al < \frac 12$, $N \les N_2$, $N_1 \les N_2$, $\dl, s > 0$ are sufficiently small, and summing over dyadic $N, N_1, N_2 \geq 1$, we obtain the desired inequality \eqref{vz}.

\medskip \noi
\textbf{Subcase 2.2:} $N \gg N_2$.

In this subcase, note that due to \eqref{K_bdd}, we can assume that $\jb{ (\tau - |n|^2) - (\tau_1 - |n_1|^2) + (\tau_2 - |n_2|^2)} \les N_2^{10}$, since otherwise we can conclude by using similar steps as in Case 1. Similarly, we can assume that $\jb{\tau_1} \les N_2^{10}$ and also $\jb{\tau_2} \les N_2^{10}$. Thus, we have $\big| |n|^2 - |n_1|^2 + |n_2|^2 \big| \les N_2^{10}$, so that $\big| |n|^2 - |n_1|^2 \big| \les N_2^{10}$.

We now perform an orthogonality argument. Note that we have $N_1 \sim N \gg N_2$ in this subcase. We decompose the set $\{ |n| \sim N \}$ into balls of radius $\sim N_2$ and denote the set of these balls as $\mathcal{J}$, and we decompose the set $\{ |n_1| \sim N_1 \}$ into balls of radius $\sim N_2$ and denote the set of these balls as $\mathcal{J}_1$. Note that for each fixed $J \in \mathcal{J}$, the product $\ind_J (n) \cdot \ind_{J_1}(n_1)$ is non-zero for at most a fixed constant number of $J_1 \in \mathcal{J}_1$, and we denote the set of these $J_1$'s as $\mathcal{J}_1 (J)$. Thus, by the Cauchy-Schwarz inequalities in $\tau$ and $n$, \eqref{K_bdd}, and Minkowski's inequality, we have
\begin{align}
\eqref{setup1} &\les T^\dl \sup_{\| \wt w \|_{\l_n^2 L_\tau^2} \leq 1} \sum_{\substack{N, N_1, N_2 \geq 1 \\ \textup{dyadic}}}  \sum_{J \in \mathcal{J}}  \sum_{J_1 \in \mathcal{J}_1 (J)}  N^s  N_2^{30\dl} \nonumber \\
&\quad \times \bigg| \sum_{\substack{n, n_1, n_2 \in \Z^2 \\ n_1 - n_2 = n \neq 0}} \int \int \int K \big( \tau, |n|^2 + (\tau_1 - |n_1|^2) - (\tau_2 - |n_2|^2) \big)  \jb{\tau}^{\frac 12 - \dl} \nonumber \\
&\quad \times \wt{P_{J_1} v} (\tau_1, n_1) \frac{\cj{g_{n_2}} (\o)}{\jb{n_2}^{1 - \al}} \cj{\ft \chi} (\tau_2) \cj{\wt{P_J w}}(\tau, n) \, d\tau d\tau_1 d\tau_2 \bigg| \nonumber \\
&\les T^\dl \sup_{\| \wt w \|_{\l_n^2 L_\tau^2} \leq 1} \sum_{\substack{N, N_1, N_2 \geq 1 \\ \textup{dyadic}}} \sum_{J \in \mathcal{J}}  \sum_{J_1 \in \mathcal{J}_1 (J)} N^s  N_2^{30\dl} \| P_J w \|_{\l_n^2 L_\tau^2} \nonumber \\
&\quad \times \bigg[ \int \jb{\tau}^{-1 - 2\dl} \bigg( \sum_{m \in \Z} \int \int \jb{\tau - \tau_1 + \tau_2 - m}^{-1} \jb{\tau_1}^{- \frac 12 - \dl} \cj{\ft \chi} (\tau_2) \nonumber \\
&\quad \times \bigg\| \sum_{n_1, n_2 \in \Z^2} h^m_{n n_1 n_2} \ind_{S_2} \cdot \frac{\cj{g_{n_2}} (\o)}{\jb{n_2}^{1 - \al}}  \jb{\tau_1}^{\frac 12 + \dl} \wt{P_{J_1} v} (\tau_1, n_1) \bigg\|_{\l_n^2} d\tau_1 d\tau_2 \bigg)^2 d\tau \bigg]^{1/2},
\label{vz3}
\end{align}

\noi
where $h_{n n_1 n_2}^m$ is the base tensor as defined in \eqref{defh} and $S_2$ is a set defined by
\begin{align*}
S_2 :\!&= S_2 (N_2, J, J_1) \\
&= \{ (n, n_1, n_2) \in (\Z^2)^3: n \neq 0, |n|^2 - |n_1|^2 \les N_2^{10}, n \in J, n_1 \in J_1, |n_2| \sim N_2 \}. 
\end{align*}

\noi
Note that for $(n, n_1, n_2)$ restricted in $S_2$, we have $\les N_2^{10}$ choices for the value
\[ m = |n|^2 - |n_1|^2 + |n_2|^2, \]

\noi
which implies that
\begin{align}
\sum_{m \in \Z} \jb{\tau - \tau_1 + \tau_2 - m}^{-1} \les \log (1 + N_2^{10}) \les N_2^\dl.
\label{summ2}
\end{align}

\noi
Thus, continuing with \eqref{vz3}, by using $N_1 \sim N$, H\"older's inequalities in $m$, $J$, and $J_1$, \eqref{summ2}, and the Cauchy-Schwarz inequalities in $\tau_1$, $J$, and $N_1 \sim N$, we obtain
\begin{align}
\eqref{setup1} &\les T^\dl \sup_{\| \wt w \|_{\l_n^2 L_\tau^2} \leq 1} \sum_{\substack{N, N_1, N_2 \geq 1 \\ \textup{dyadic}}}  \sum_{J \in \mathcal{J}}  \sum_{J_1 \in \mathcal{J}_1 (J)}    N_2^{31\dl} \| P_J w \|_{\l_n^2 L_\tau^2} \nonumber \\
&\quad \times \sup_{\substack{m \in \Z \\ J \in \mathcal{J} \\ J_1 \in \mathcal{J}_1 (J)}} \bigg\| \sum_{n_2 \in \Z^2} h^m_{n n_1 n_2} \ind_{S_2} \cdot \frac{\cj{g_{n_2}} (\o)}{\jb{n_2}^{1 - \al}} \bigg\|_{n \to n_1} \| P_{J_1} v \|_{X^{s, \frac 12 + \dl}} \nonumber \\
&\les T^\dl \sum_{\substack{N_2 \geq 1 \\ \textup{dyadic}}} N_2^{31 \dl} \sup_{\substack{m \in \Z \\ J \in \mathcal{J} \\ J_1 \in \mathcal{J}_1 (J)}} \bigg\| \sum_{n_2 \in \Z^2} h^m_{n n_1 n_2} \ind_{S_2} \cdot \frac{\cj{g_{n_2}} (\o)}{\jb{n_2}^{1 - \al}} \bigg\|_{n \to n_1} \| v \|_{X^{s, \frac 12 + \dl}}.
\label{vz4}
\end{align}

\noi
By Lemma \ref{LEM:rten2}, the Gaussian tail bound \eqref{gauss}, and Lemma \ref{LEM:dten}, we have
\begin{align}
\bigg\| &\sum_{n_2 \in \Z^2} h^m_{n n_1 n_2} \ind_{S_2} \cdot \frac{\cj{g_{n_2}} (\o)}{\jb{n_2}^{1 - \al}} \bigg\|_{n \to n_1} \nonumber \\
&\les T^{-2 \ta} N_2^{-1 + 2\dl + \al} \max \big\{ \| h^m_{n n_1 n_2} \ind_{S_2} \|_{n n_2 \to n_1}, \| h^m_{n n_1 n_2} \ind_{S_2} \|_{n \to n_1 n_2} \big\} \nonumber \\
&\les T^{-2 \ta} N_2^{-\frac 12 + 2\dl + \al}
\label{ten2}
\end{align}

\noi
outside an exceptional set of probability $\leq C \exp(-c N_2^\dl / T^\ta)$ for some universal constants $C, c > 0$. Thus, combining \eqref{vz4} and \eqref{ten2}, using the fact that $\al < \frac 12$ and $\dl, s > 0$ are sufficiently small, and summing over dyadic $N_2 \geq 1$, we obtain the desired inequality \eqref{vz}.

\medskip \noi
(ii) This part follows similarly from part (i), so that we will be brief here. Using similar steps as in Case 1 of part (i), we can assume that $\jb{\tau} \les N_1^{10}$.

When $N \les N_1$, we use the Cauchy-Schwarz inequalities in $\tau$ and $n$, \eqref{K_bdd}, Minkowski's inequality, and H\"older's inequality in $m$ to obtain
\begin{align}
\eqref{setup1} &\les T^\dl \sum_{\substack{N, N_1, N_2 \geq 1 \\ \textup{dyadic}}} N^s  N_1^{31\dl} \bigg[ \int \jb{\tau}^{-1 - 2\dl}  \bigg( \int \int  \jb{\tau_2}^{- \frac 12 - \dl} \ft \chi (\tau_1) \nonumber \\
&\quad \times \sup_{m \in \Z} \bigg\| \sum_{n_1, n_2 \in \Z^2} h^m_{n n_1 n_2} \ind_{S_3} \cdot \frac{g_{n_1} (\o)}{\jb{n_1}^{1 - \al}}  \jb{\tau_1}^{\frac 12 + \dl} \cj{\wt{P_{N_2} v}} (\tau_2, n_2) \bigg\|_{\l_n^2} d\tau_1 d\tau_2 \bigg)^2 d\tau \bigg]^{1/2},
\label{zv1}
\end{align}

\noi
where $h^m_{n n_1 n_2}$ is the base tensor as defined in \eqref{defh} and $S_3$ is a set defined by
\begin{align*}
S_3 : \! &= S_3 (N, N_1, N_2) \\
&= \{ (n, n_1, n_2) \in (\Z^2)^3: n \neq 0, n_2 \neq 0, |n| \sim N, |n_1| \sim N_1, |n_2| \sim N_2 \}.
\end{align*}

\noi
Then, by \eqref{zv1}, the Cauchy-Schwarz inequality in $\tau_2$, Lemma \ref{LEM:rten}, the Gaussian tail bound, and Lemma \ref{LEM:dten}, we obtain
\begin{align*}
\eqref{setup1} &\les T^\dl \sum_{\substack{N, N_1, N_2 \geq 1 \\ \textup{dyadic}}} N^s  N_1^{31\dl} \sup_{m \in \Z} \bigg\| \sum_{n_1 \in \Z^2} h^m_{n n_1 n_2} \ind_{S_3} \cdot \frac{g_{n_1} (\o)}{\jb{n_1}^{1 - \al}} \bigg\|_{n \to n_2} \| P_{N_2} v \|_{X^{s, \frac 12 + \dl}} \\
&\les T^{\dl - 2 \ta} \sum_{\substack{N, N_1, N_2 \geq 1 \\ \textup{dyadic}}} N^s N_1^{-1 + 33\dl + \al} \sup_{m \in \Z} \max \big\{ \| h^m_{n n_1 n_2} \ind_{S_3} \|_{nn_1 \to n_2}, \| h^m_{n n_1 n_2} \ind_{S_3} \|_{n \to n_1 n_2} \big\} \\
&\les T^{\dl - 2\ta} \sum_{\substack{N, N_1, N_2 \geq 1 \\ \textup{dyadic}}} N^s N_1^{-\frac 12 + 33\dl + \al}
\end{align*}

\noi
outside an exceptional set of probability $\leq C \exp (-c N_1^\dl / T^\ta)$ for some universal constants $C, c > 0$. Thus, since $\al < \frac 12$, $N \les N_1$, $N_2 \les N_1$, and $\dl, s > 0$ are sufficiently small, we can sum over dyadic $N, N_1, N_2 \geq 1$ to obtain the desired inequality \eqref{zv}.

When $N \gg N_1$, as in Subcase 2.2 in part (i), we can assume that $\jb{ (\tau - |n|^2) - (\tau_1 - |n_1|^2) + (\tau_2 - |n_2|^2)} \les N_1^{10}$, $\jb{\tau_1} \les N_1^{10}$, and $\jb{\tau_2} \les N_1^{10}$, so that $|n|^2 + |n_2|^2 \les N_1^{10}$. We perform an orthogonality argument as in Subcase 2.2 in part (i) to decompose $\{ |n| \sim N \}$ into a set of balls (denoted as $\mathcal{J}$) of radius $\sim N_1$ and decompose $\{ |n_2| \sim N_2 \}$ into a set of balls (denoted as $\mathcal{J}_2$) of radius $\sim N_1$. For each $J \in \mathcal{J}$, $\ind_J (n) \cdot \ind_{J_2} (n_2)$ is non-zero for at most a fixed constant number of $J_2 \in \mathcal{J}_2$, and we denote the set of these $J_2$'s as $\mathcal{J}_2 (J)$. By the Cauchy-Schwarz inequalities in $\tau$ and $n$, \eqref{K_bdd}, Minkowski's inequality, H\"older's inequalities in $m$, $J$, and $J_2$, and the Cauchy-Schwarz inequalities in $\tau_2$, $J$, and $N_2 \sim N$, we have
\begin{align}
\eqref{setup1} &\les T^\dl \sup_{\| \wt w \|_{\l_n^2 L_\tau^2} \leq 1} \sum_{\substack{N, N_1, N_2 \geq 1 \\ \textup{dyadic}}}  \sum_{J \in \mathcal{J}}  \sum_{J_2 \in \mathcal{J}_2 (J)}  N^s  N_1^{30\dl} \| P_J w \|_{\l_n^2 L_\tau^2} \nonumber \\
&\quad \times \bigg[ \int \jb{\tau}^{-1 - 2 \dl} \bigg( \sum_{m \in \Z} \int \int \jb{\tau - \tau_1 + \tau_2 - m}^{-1} \jb{\tau_1}^{-\frac 12 - \dl} \ft \chi (\tau_1) \nonumber \\
&\quad \times \bigg\| \sum_{n_1, n_2 \in \Z^2} h^m_{n n_1 n_2} \ind_{S_4} \cdot \frac{g_{n_1} (\o)}{\jb{n_1}^{1 - \al}} \jb{\tau_1}^{\frac 12 + \dl} \cj{\wt{P_{J_2} v}} (\tau_2, n_2) \bigg\|_n d\tau_1 d\tau_2 \bigg)^2 d\tau \bigg]^{1/2} \nonumber \\
&\les T^\dl \sum_{\substack{N_1 \geq 1 \\ \text{dyadic}}} N_1^{31 \dl} \sup_{\substack{m \in \Z \\ J \in \mathcal{J} \\ J_2 \in \mathcal{J}_2 (J)}} \bigg\| \sum_{n_1 \in \Z^2} h^m_{n n_1 n_2} \ind_{S_4} \cdot \frac{g_{n_1} (\o)}{\jb{n_1}^{1 - \al}} \bigg\|_{n \to n_2} \| v \|_{X^{s, \frac 12 + \dl}},
\label{zv3}
\end{align}

\noi
where $h^m_{n n_1 n_2}$ is the base tensor as defined in \eqref{defh} and $S_4$ is a set defined by
\begin{align*}
S_4 :\!&= S_4(N_1, J, J_2) \\
&= \{ (n, n_1, n_2) \in (\Z^2)^3: n \neq 0, n_2 \neq 0, |n|^2 + |n_2|^2 \les N_1^{10}, n \in J, |n_1| \sim N_1, n_2 \in J_2 \}.
\end{align*}

\noi
By Lemma \ref{LEM:rten2}, the Gaussian tail bound, and Lemma \ref{LEM:dten}, we have
\begin{align}
\bigg\| &\sum_{n_1 \in \Z^2} h^m_{n n_1 n_2} \ind_{S_4} \cdot \frac{g_{n_1} (\o)}{\jb{n_1}^{1 - \al}} \bigg\|_{n \to n_2} \nonumber \\
&\les T^{-2 \ta} N_1^{-1 + 2\dl + \al} \max \big\{ \|  h^m_{n n_1 n_2} \ind_{S_4} \|_{n n_1 \to n_2}, \|  h^m_{n n_1 n_2} \ind_{S_4} \|_{n \to n_1 n_2} \big\} \nonumber \\
&\les T^{-2 \ta} N_1^{-\frac 12 + 2\dl + \al}
\label{ten4}
\end{align}

\noi
outside an exceptional set of probability $\leq C \exp(-c N_1^\dl / T^\ta)$ for some universal constants $C, c > 0$. Thus, since $\al < \frac 12$ and $\dl, s > 0$ are sufficiently small, we can combine \eqref{zv3} and \eqref{ten4} and sum over dyadic $N_1 \geq 1$ to obtain the desired inequality \eqref{zv}.

\medskip \noi
(iii) We consider the following two cases.

\medskip \noi
\textbf{Case 1:} $\jb{\tau} \gg N_1^{10} N_2^{10}$.

In this case, by H\"older's inequalities in $n_1$ and $n_2$, \eqref{K_bdd}, and the Cauchy-Schwarz inequality in $\tau$, we have
\begin{align*}
\begin{split}
\eqref{setup1} &\les T^\dl \sup_{\| \wt w \|_{\l_n^2 L_\tau^2} \leq 1} \sum_{\substack{N, N_1, N_2 \geq 1 \\ \textup{dyadic}}} N^s N_1^{-10} N_2^{-10} N_1^2 N_2^2 N_1^{-1 + \al} N_2^{-1 + \al} \\
&\qquad \times \sup_{\substack{ n_1 \in \Z^2 \\ \jb{n_1} \sim N_1 }}  \sup_{\substack{ n_2 \in \Z^2 \\ \jb{n_2} \sim N_2 }} |g_{n_1}(\o)| |g_{n_2}(\o)|  \int \frac{\big| \cj{\wt{P_N w}} (n_1 - n_2, \tau) \big|}{\jb{\tau - |n_1 - n_2|^2 + |n_1|^2 - |n_2|^2}}   \,d\tau \\
&\les T^\dl \sum_{\substack{N, N_1, N_2 \geq 1 \\ \textup{dyadic}}} N^s N_1^{-9 + \al} N_2^{-9 + \al} \sup_{\substack{ n_1 \in \Z^2 \\ \jb{n_1} \sim N_1 }}  \sup_{\substack{ n_2 \in \Z^2 \\ \jb{n_2} \sim N_2 }} |g_{n_1}(\o)| |g_{n_2}(\o)|.
\end{split}
\end{align*}

\noi
Without loss of generality, we can assume that $N_1 \leq N_2$. By using the following Gaussian tail bounds:
\begin{align}
\sum_{\substack{ n_1 \in \Z^2 \\ \jb{n_1} \sim N_1 }} P(|g_{n_1}| &> T^{-\ta} N_2^\dl) < C \exp\Big(-c \frac{N_2^\dl}{T^{\ta}} \Big), \label{gauss1} \\
\sum_{\substack{ n_2 \in \Z^2 \\ \jb{n_2} \sim N_2 }} P(|g_{n_2}| &> T^{-\ta} N_2^\dl) < C \exp\Big(-c \frac{N_2^\dl}{T^{\ta}} \Big),
\label{gauss2}
\end{align}

\noi
we obtain
\[ \eqref{setup1} \les T^{\dl - \ta} \sum_{\substack{N, N_1, N_2 \geq 1 \\ \textup{dyadic}}} N^s N_1^{-9 + \al} N_2^{-9 + 2\dl + \al}. \]

\noi
outside an exceptional set of probability $\leq C \exp(-c N_2^\dl / T^\ta)$. Note that we have $N \les N_2$. Thus, since $\al < \frac 12$ and $\dl, s > 0$ are sufficiently small, we can sum over dyadic $N, N_1, N_2 \geq 1$ to obtain \eqref{zz}.

\medskip \noi
\textbf{Case 2:} $\jb{\tau} \les N_1^{10} N_2^{10}$.

In this case, by the Cauchy-Schwarz inequalities in $\tau$ and $n$, \eqref{K_bdd}, and Minkowski's inequality, we have
\begin{align}
\eqref{setup1} &\les T^\dl \sup_{\| \wt w \|_{\l_n^2 L_\tau^2} \leq 1}  \sum_{\substack{N, N_1, N_2 \geq 1 \\ \textup{dyadic}}}  N^s N_1^{30\dl} N_2^{30 \dl} \nonumber \\
&\quad \times \bigg| \sum_{\substack{n, n_1, n_2 \in \Z^2 \\ n_1 - n_2 = n \neq 0}} \int \int \int K \big( \tau, |n|^2 + (\tau_1 - |n_1|^2) - (\tau_2 - |n_2|^2) \big) \jb{\tau}^{\frac 12 - \dl} \nonumber \\
&\quad \times \frac{g_{n_1} (\o)}{\jb{n_1}^{1 - \al}} \ft \chi (\tau_1) \frac{\cj{g_{n_2}} (\o)}{\jb{n_2}^{1 - \al}} \cj{\ft \chi} (\tau_2) \cj{\wt{P_N w}} (\tau, n) \, d\tau d\tau_1 d\tau_2 \bigg| \nonumber \\
&\les T^\dl \sum_{\substack{N, N_1, N_2 \geq 1 \\ \textup{dyadic}}}  N^s N_1^{30\dl} N_2^{30 \dl} \bigg[ \int \jb{\tau}^{-1 - 2\dl} \nonumber \\
&\quad \times \bigg( \sum_{m \in \Z} \int \int \jb{\tau - \tau_1 + \tau_2 - m}^{-1} \ft \chi (\tau_1) \cj{\ft \chi} (\tau_2) \nonumber \\
&\quad \times \bigg\| \sum_{n_1, n_2 \in \Z^2} h^m_{n n_1 n_2} \ind_{S_1} \cdot \frac{g_{n_1} (\o)}{\jb{n_1}^{1 - \al}} \frac{\cj{g_{n_2}} (\o)}{\jb{n_2}^{1 - \al}} \bigg\|_{\l_n^2} \, d\tau_1 d\tau_2 \bigg)^2 d\tau \bigg]^{1/2},
\label{zz1}
\end{align}

\noi
where $h^m_{n n_1 n_2}$ is the base tensor as defined in \eqref{defh} and $S_1$ is as defined in \eqref{s1}. Note that for $(n, n_1, n_2)$ restricted in $S_1$, we have $\les \max \{ N_1^2, N_2^2 \}$ choices for the value
\[ m = |n|^2 - |n_1|^2 + |n_2|^2, \]

\noi
which implies that
\begin{align}
\sum_{m \in \Z} \jb{\tau - \tau_1 + \tau_2 - m}^{-1} \les \log (1 + N_1^2 N_2^2) \les N_1^\dl N_2^\dl.
\label{summ3}
\end{align}

\noi
Again, we can assume without loss of generality that $N_1 \leq N_2$. Thus, continuing with \eqref{zz1}, by H\"older's inequality in $m$, \eqref{summ3}, Lemma \ref{LEM:rten}, the Gaussian tail bounds \eqref{gauss1} and \eqref{gauss2}, and Lemma \ref{LEM:dten}, we have
\begin{align*}
\eqref{setup1} &\les T^\dl \sum_{\substack{N, N_1, N_2 \geq 1 \\ \textup{dyadic}}}  N^s N_1^{31 \dl} N_2^{31 \dl} \sup_{m \in \Z} \bigg\| \sum_{n_1, n_2 \in \Z^2} h^m_{n n_1 n_2} \ind_{S_1} \cdot \frac{g_{n_1} (\o)}{\jb{n_1}^{1 - \al}} \frac{\cj{g_{n_2}} (\o)}{\jb{n_2}^{1 - \al}} \bigg\|_n \\
&\les T^{\dl - 2\ta} \sum_{\substack{N, N_1, N_2 \geq 1 \\ \textup{dyadic}}}  N^s N_1^{-1 + 31 \dl + \al} N_2^{-1 + 34 \dl + \al} \sup_{m \in \Z} \| h^m_{n n_1 n_2} \ind_{S_1} \|_{n n_1 n_2} \\
&\les T^{\dl - 2\ta} \sum_{\substack{N, N_1, N_2 \geq 1 \\ \textup{dyadic}}}  N^s N_1^{-\frac 12 + 31 \dl + \al} N_2^{-\frac 12 + 35 \dl + \al}
\end{align*}

\noi
outside an exceptional set of probability $\leq C \exp(-c N_2^\dl / T^\ta)$ for some universal constants $C, c > 0$. Note that we have $N \les N_2$. Thus, since $\al < \frac 12$ and $\dl, s > 0$ are sufficiently small, we can sum over dyadic $N, N_1, N_2 \geq 1$ to obtain the desired inequality \eqref{zz}.
\end{proof}

\begin{remark} \rm
\label{RMK:zv}
The frequency projections $P_{\neq 0}$ in all three parts of Proposition \ref{PROP:zz} are necessary in our approach. For \eqref{vz} and \eqref{zv}, we need to avoid the zeroth frequencies due to the necessity of the condition $n \neq 0$ in the base tensor \eqref{defh} (see also Remark \ref{RMK:ten}). For \eqref{zz}, without the frequency projection $P_{\neq 0}$, one can show that the zeroth frequency diverges almost surely when $\al \geq 0$ using the argument as in the proof of Proposition 1.6 in \cite{OO}, so that we need to remove the zeroth frequency.

Furthermore, in part (ii) of Proposition \ref{PROP:zz}, the assumption that $v$ has mean zero is important for us to obtain the desired estimate. Without this assumption, i.e. when $v$ is allowed to be a non-zero constant, the LHS of \eqref{zv} essentially becomes $\| \chi \cdot z \|_{X^{s, \frac 12 + \dl}}$, which is equal to infinity almost surely when $\al \geq 0$.
\end{remark}

\section{Proof of Theorem \ref{THM:LWP}}
\label{SEC:well}

In this section, we prove Theorem \ref{THM:LWP}, the almost sure local well-posedness result of the quadratic NLS \eqref{qNLS}. We fix $\al < \frac 12$ throughout this section.

We recall from \eqref{exp} the following first order expansion:
\[ u = z + v. \]

\noi
Here, $z$ is the random linear solution as in $\eqref{defz}$ and $v$ is the remainder term that satisfies \eqref{vNLS}, which we can write in the following Duhamel formulation:
\begin{align}
v (t) = \G [v] (t) := -i \I_\chi \bigg( |z + v|^2 - \fint |z + v|^2 \bigg) (t),
\label{Duh}
\end{align}

\noi
where $0 < t \leq 1$ and $\I_\chi$ is the Duhamel operator as defined in \eqref{defI} with $\chi$ being a smooth cut-off function such that $\chi \equiv 1$ on $[-1, 1]$ and $\chi \equiv 0$ outside of $[-2, 2]$. We note from \eqref{Duh} that $v$ has mean zero (i.e. has no zeroth frequency term). We show that $\G$ is a contraction map on a ball of the space $X_T^{s, b} \subset C( [-T, T]; H^s(\T^2) )$ for some $s > 0$ and $b > \frac 12$ outside an exceptional set of exponentially small probability.

Let $s, \dl > 0$ be sufficiently small. Let $\varphi$ be an arbitrary smooth function with $\varphi \equiv 1$ on $[-1, 1]$ and $\varphi \equiv 0$ outside of $[-2, 2]$, and let $\varphi_T (t) = \varphi(t / T)$ for $0 < T \leq 1$. By the definition of $X_T^{s,b}$-norm, \eqref{Duh}, Proposition \ref{PROP:vv}, and Proposition \ref{PROP:zz}, we have that for every $0 < T \leq 1$,
\begin{align*}
\begin{split}
\big\| \G [v]  \big\|_{X_T^{s, \frac 12 + \dl}} &\leq \big\| P_{\neq 0} \big( \varphi_T \cdot \I_\chi \big(  |\chi \cdot z + v|^2 \big) \big) \big\|_{X^{s, \frac 12 + \dl}} \\
&\leq \big\| \varphi_T \cdot \I_\chi \big( | v |^2 \big) \big\|_{X^{s, \frac 12 + \dl}} + \big\| P_{\neq 0} \big( \varphi_T \cdot \I_\chi \big( v \cdot  \cj{\chi \cdot z} \big) \big) \big\|_{X^{s, \frac 12 + \dl}} \\
&\quad + \big\| P_{\neq 0} \big( \varphi_T \cdot \I_\chi \big( \chi \cdot z \cdot \cj{ v }  \big) \big) \big\|_{X^{s, \frac 12 + \dl}} + \big\| P_{\neq 0} \big( \varphi_T \cdot \I_\chi \big( | \chi \cdot z |^2 \big) \big) \big\|_{X^{s, \frac 12 + \dl}} \\
&\les T^{\dl - \ta} \Big( \| v \|_{X^{s, \frac 12 + \dl}}^2  +  2\| v \|_{X^{s, \frac 12 + \dl}}  + 1  \Big),
\end{split}
\end{align*}

\noi
outside an exceptional set of probability $\leq C \exp(- \frac{c}{T^\ta})$ with $C, c > 0$ being constants and $0 < \ta \ll \dl$. Taking the infimum over all extensions of $v$ outside the time interval $[-T, T]$, we obtain
\[ \big\| \G [v] \big\|_{X_T^{s, \frac 12 + \dl}} \les T^{\frac{\dl}{2}} \Big( \| v \|_{X_T^{s, \frac 12 + \dl}} + 1 \Big)^2. \]

\noi
Similarly, we obtain the following difference estimate outside an exceptional set of probability $\leq C \exp(- \frac{c}{T^\ta})$:
\begin{align*}
\big\| \G[v_1] - \G[v_2] \big\|_{X_T^{s, \frac 12 + \dl}} \les T^{\frac{\dl}{2}} \| v_1 - v_2 \|_{X_T^{s, \frac 12 + \dl}} \Big( \| v_1 \|_{X_T^{s, \frac 12 + \dl}} + \| v_2 \|_{X_T^{s, \frac 12 + \dl}} + 1 \Big).
\end{align*}

\noi
Therefore, for a fixed $R > 0$, by choosing $T = T(R) > 0$ sufficiently small, we obtain that $\G$ is a contraction on the ball $B_R \subset X_T^{s, \frac 12 + \dl}$ of radius $R$ outside an exceptional set of probability $\leq C \exp(- \frac{c}{T^\ta})$. This finishes the proof of Theorem \ref{THM:LWP}.

\section{Proof of Proposition \ref{PROP:ill}}
\label{SEC:ill}

In this section, we prove Proposition \ref{PROP:ill}, the non-convergence of the Picard second iterate $z_N^{(2)}$ as defined in \eqref{defzN2}.

We fix $n \neq 0$, $t \neq 0$, and $\al \geq \frac 34$. Let us first show that $\lim_{N \to \infty} \E \big[ | \F_x z_N^{(2)} (t, n) |^2 \big] = \infty$. A direct computation yields
\begin{align}
\F_x z_N^{(2)} (t, n) &= \int_0^t e^{- i (t - t') |n|^2} \sum_{\substack{k \in \Z^2 \\  |k| \leq N \\ |n + k| \leq N}} e^{-i t' |n + k|^2 + i t' |k|^2} \frac{g_{n + k} (\o) \cj{g_k} (\o)}{\jb{n + k}^{1 - \al} \jb{k}^{1 - \al}} \, dt' \nonumber \\
&= \sum_{\substack{k \in \Z^2 \\  |k| \leq N \\ |n + k| \leq N}}  \frac{g_{n + k} (\o) \cj{g_k} (\o)}{\jb{n + k}^{1 - \al} \jb{k}^{1 - \al}}  e^{-it |n|^2} \frac{1 - e^{-2i t n \cdot k}}{2i n \cdot k}.
\label{FxzN}
\end{align}

\noi
By independence, we can compute that
\begin{align}
\E \big[ | \F_x z_N^{(2)} (t, n) |^2 \big] = \sum_{\substack{k \in \Z^2 \\  |k| \leq N \\ |n + k| \leq N}} \frac{1}{\jb{n + k}^{2 - 2 \al} \jb{k}^{2 - 2 \al}} \frac{2 \sin (t n \cdot k)^2 }{|n \cdot k|^2}.
\label{En}
\end{align}

\noi
We focus on the case when $n \cdot k = 0$, so that \eqref{En} is bounded from below (up to some constant depending only on $n$ and $t$) by
\begin{align}
\sum_{\substack{k \in \Z^2 \\ n \cdot k = 0 \\  |k| \leq N }} \frac{1}{ \jb{k}^{4 - 4\al} }.
\label{sumk}
\end{align}

\noi
We write $n = (n_1, n_2)$. Note that if either $n_1 = 0$ or $n_2 = 0$, then we can easily see that \eqref{sumk} diverges as $N \to \infty$ when $\al \geq \frac 34$. If $n_1 \neq 0$ and $n_2 \neq 0$, we note that all $k$'s that satisfy $n \cdot k = 0$ are of the form $k = ak'$, where $a \in \Z$ and
\[ k' = \Big( - \frac{n_2}{ \gcd (n_1, n_2) }, \frac{n_1}{ \gcd (n_1, n_2) } \Big). \]

\noi
Thus, \eqref{sumk} is bounded from below by
\[ \sum_{\substack{a \in \Z \\ 0 < |a| \leq N / |k'|}} \frac{1}{|a|^{4 - 4\al} \jb{k'}^{4 - 4\al}}, \]
which increases to infinity as $N \to \infty$ when $\al \geq \frac 34$. This shows that 
\begin{align}
\E \big[ | \F_x z_N^{(2)} (t, n) |^2 \big] \longrightarrow \infty
\label{L2infty}
\end{align}

\noi
as $N \to \infty$.

We now show that, for any sequence $\{N_\l\}_{\l \in \N} \subset \N$, the sequence of random variables $\{ \F_x z_{N_\l}^{(2)} (t, n) \}_{\l \in \N}$ is not tight. Assume for the sake of contradiction that $\{ \F_x z_{N_\l}^{(2)} (t, n) \}_{\l \in \N}$ is tight. Using the explicit formula of $\F_x z_{N_\l}^{(2)} (t, n)$ in \eqref{FxzN}, we can write $\F_x z_{N_\l}^{(2)} (t, n) = X_\l + i Y_\l$, where $X_\l, Y_\l \in \H_{\leq 2}$ are real-valued. Here, we recall that the space $\H_{\leq 2}$ is as defined in \eqref{Hk}. By Lemma \ref{LEM:wie}, we have
\begin{align}
\E \big[ | \F_x z_{N_\l}^{(2)} (t, n) |^4 \big]^{\frac 14} &\leq \E \big[ |X_\l|^4 \big]^{\frac 14} + \E \big[ |Y_\l|^4 \big]^{\frac 14} \nonumber \\
&\leq 3\E \big[ |X_\l|^2 \big]^{\frac 12} + 3 \E \big[ |Y_\l|^2 \big]^{\frac 12} \nonumber \\
&\leq 3\sqrt{2} \E \big[ |\F_x z_{N_\l}^{(2)} (t, n) |^2 \big]^{\frac 12}.
\label{L4bdd}
\end{align}

\noi
By the Paley-Zygmund inequality and \eqref{L4bdd}, we have
\begin{align}
P \bigg( | \F_x z_{N_\l}^{(2)} (t, n) |^2 > \frac{\E \big[ | \F_x z_{N_\l}^{(2)} (t, n) |^2 \big]}{2} \bigg) &\geq \frac{1}{4} \frac{\big( \E \big[ | \F_x z_{N_\l}^{(2)} (t, n) |^2 \big] \big)^2}{\E \big[ | \F_x z_{N_\l}^{(2)} (t, n) |^4 \big]} \geq \frac{1}{1296}.
\label{ill2}
\end{align}

\noi
By tightness, we know that there exists a constant $A > 0$ such that for all $\l \in \N$, 
\begin{align}
P \big( \big| \F_x z_{N_\l}^{(2)} (t, n) \big| > A \big) < \frac{1}{1296}.
\label{ill3}
\end{align}

\noi
Due to \eqref{ill2} and \eqref{ill3}, we must have $\E \big[ | \F_x z_{N_\l}^{(2)} (t, n) |^2 \big] \leq 2 A^2$ for all $\l \in \N$, which is a contradiction to \eqref{L2infty}. Therefore, the sequence $\{ \F_x z_{N_\l}^{(2)} (t, n) \}_{\l \in \N}$ is not tight. This finishes the proof of Proposition \ref{PROP:ill}.

\begin{remark} \rm
In the proof above, although we only considered the case when $n \cdot k = 0$, we point out that the range $\al \geq \frac 34$ for the divergence of $\E \big[ | \F_x z_N^{(2)} (t, n) |^2 \big]$ is sharp. More precisely, suppose that we have $\al < \frac 34$. Note that the RHS of \eqref{En} converges as $N \to \infty$ if and only if the following integral converges:
\begin{align}
\int_{\{ x \in \R^2: |x| \leq N \}} \frac{1}{\jb{x}^{4 - 4 \al}} \frac{ \sin (t n \cdot x)^2 }{ |n \cdot x|^2 } \,dx.
\label{int}
\end{align}

\noi
By using a change of variable, we note that the convergence of \eqref{int} is equivalent to the convergence of the following term:
\begin{align*}
\int_0^N \int_0^N \frac{1}{(1 + |y_1|^2 + |y_2|^2)^{2 - 2\al}} \frac{ \sin(t y_1)^2 }{|y_1|^2} \,dy_1 dy_2,
\end{align*}

\noi
which can easily be seen to converge when $\al < \frac 34$.
\end{remark}

\begin{ackno} \rm
The author would like to thank his advisor, Tadahiro Oh, for suggesting this problem and for his support throughout the entire work. Also, the author would like to thank Bjoern Bringmann and Yuzhao Wang for pointing out a mistake in a previous version of the work. In addition, the author is grateful to Guangqu Zheng and Younes Zine for helpful suggestions and discussions. R.L. was supported by the European Research Council (grant no. 864138 ``SingStochDispDyn'').
\end{ackno}


\begin{thebibliography}{99}

\bibitem{BOP3}
\'A.~B\'enyi, T.~Oh, O.~Pocovnicu,
{\it Higher order expansions for the probabilistic local Cauchy theory of the cubic nonlinear Schrödinger equation on $\R^3$}, Trans. Amer. Math. Soc. Ser. B 6 (2019), 114--160.

\bibitem{Bour93}
J.~Bourgain,
{\it Fourier transform restriction phenomena for certain lattice subsets and applications to nonlinear evolution equations, I: Schr\"odinger equations}, Geom. Funct. Anal. 3 (1993), 107--156.

\bibitem{Bour95}
J.~Bourgain, 
{\it Nonlinear Schr\"odinger equations}, Hyperbolic equations and frequency interactions (Park City,
UT, 1995), 3--157, IAS/Park City Math. Ser., 5, Amer. Math. Soc., Providence, RI, 1999.

\bibitem{Bour96}
J.~Bourgain, 
{\it Invariant measures for the 2D-defocusing nonlinear Schr\"odinger equation}, Comm. Math. Phys., 176 (1996), no. 2, 421--445.

\bibitem{Bour97}
J.~Bourgain,
{\it Invariant measures for the Gross-Piatevskii equation}, J. Math. Pures Appl. (9) 76 (1997), no. 8, 649--702.

\bibitem{Bri}
B.~Bringmann,
{\it Invariant Gibbs measures for the three-dimensional wave equation with a Hartree nonlinearity II: dynamics}, arXiv:2009.04616v3 [math.AP].

\bibitem{BDNY}
B.~Bringmann, Y.~Deng, A.~Nahmod, H.~Yue,
{\it Invariant Gibbs measures for the three-dimensional cubic nonlinear wave equation}, arXiv:2205.03893v1 [math.AP].

\bibitem{BT08}
N.~Burq, N.~Tzvetkov,
{\it Random data Cauchy theory for supercritical wave equations. I. Local theory}, 
Invent. Math., 173 (2008), no. 3, 449--475.

\bibitem{CO}
J.~Colliander, T.~Oh, 
{\it Almost sure well-posedness of the cubic nonlinear Schr\"odinger equation below $\ell^2(\T)$}, Duke Math. J., 161 (2012), no. 3, 367--414.

\bibitem{DPD}
G.~Da Prato, A.~Debussche, 
{\it Strong solutions to the stochastic quantization equations}, Ann. Probab., 31 (2003), no. 4, 1900--1916.

\bibitem{DNY1}
Y.~Deng, A.~Nahmod, H.~Yue,
{\it Optimal local well-posedness for the periodic derivative nonlinear Schrödinger equation}, Comm. Math. Phys. 384 (2021), no. 2, 1061--1107.

\bibitem{DNY2}
Y.~Deng, A.~Nahmod, H.~Yue,
{\it Invariant Gibbs measures and global strong solutions for nonlinear Schr\"odinger equations in dimension two}, arXiv:1910.08492 [math.AP].

\bibitem{DNY3}
Y.~Deng, A.~Nahmod, H.~Yue, 
{\it Random tensors, propagation of randomness, and nonlinear dispersive equations}, Invent. Math. 228 (2022), no. 2, 539-686.

\bibitem{Dur}
R.~Durrett,
{\it Probability-theory and examples}, Fifth edition. Cambridge Series in Statistical and Probabilistic Mathematics, 49. Cambridge University Press, Cambridge, 2019. xii+419 pp.

\bibitem{FOSW}
C.~Fan, Y.~Ou, G.~Staffilani, H.~Wang,
{\it 2D-defocusing nonlinear Schr\"odinger equation with random data on irrational tori}, Stoch. Partial Differ. Equ. Anal. Comput. 9 (2021), no. 1, 142--206.

\bibitem{GTV}
J.~Ginibre, Y.~Tsutsumi, G.~Velo,
{\it On the Cauchy problem for the Zakharov system}, J. Funct. Anal., 151 (1997), 384--436.

\bibitem{GKO2}
M.~Gubinelli, H.~Koch, T.~Oh,
{\it Paracontrolled approach to the three-dimensional stochastic nonlinear wave equation with quadratic nonlinearity}, arXiv:1811.07808 [math.AP]

\bibitem{Hair13}
M.~Hairer,
{\it Solving the KPZ equation},
Ann. of Math. 178 (2013), no. 2, 559--664.

\bibitem{Hair14}
M.~Hairer,
{\it A theory of regularity structures},
Invent. Math. 198 (2014), no. 2, 269--504.

\bibitem{Hair14s}
M.~Hairer,
{\it Singular stochastic PDEs},
Proceedings of the International Congress of Mathematicians-Seoul 2014. Vol. 1, 685--709, Kyung Moon Sa, Seoul, 2014.

\bibitem{Hair15}
M.~Hairer,
{\it Introduction to regularity structures},
Braz. J. Probab. Stat. 29 (2015), no. 2, 175--210.

\bibitem{Hosh}
M.~Hoshino,
{\it KPZ equation with fractional derivatives of white noise}, 
Stoch. Partial Differ. Equ. Anal. Comput. 4 (2016), no. 4, 827--890.

\bibitem{Kish19}
N.~Kishimoto, 
{\it A remark on norm inflation for nonlinear Schr\"odinger equations},
Commun. Pure Appl. Anal. 18 (2019), no. 3, 1375--1402.

\bibitem{KM93}
S.~Klainerman, M.~Machedon,
{\it Space-time estimates for null forms and the local existence theorem}, Comm. Pure Appl. Math., 46 (1993), 1221--1268.

\bibitem{LO}
R.~Liu, T.~Oh,
{\it Sharp local well-posedness of the two-dimensional periodic nonlinear Schr\"odinger equation with a quadratic nonlinearity $|u|^2$}, arXiv:2203.15389 [math.AP].

\bibitem{McKean}
H.P.~McKean, 
{\it Statistical mechanics of nonlinear wave equations. IV. Cubic Schr\"odinger}, Comm. Math. Phys., 168 (1995), no. 3, 479--491. 
{\it Erratum: Statistical mechanics of nonlinear wave equations. IV. Cubic Schr\"odinger}, Comm. Math. Phys., 173 (1995), no. 3, 675.


\bibitem{Nel}
E.~Nelson,
{\it A quartic interaction in two dimensions}, 1966 Mathematical Theory of Elementary Particles (Proc. Conf., Dedham, Mass., 1965) pp. 69--73 M.I.T. Press, Cambridge, Mass.

\bibitem{NP}
I.~Nourdin, G.~Peccati,
{\it Normal approximations with Malliavin calculus. From Stein's method to universality}, Cambridge Tracts in Mathematics, 192. Cambridge University Press, Cambridge, 2012. xiv+239 pp. ISBN: 978-1-107-01777-1.

\bibitem{OO}
T.~Oh, M.~Okamoto,
{\it Comparing the stochastic nonlinear wave and heat equations: a case study}, Electron. J. Probab. 26 (2021), Paper No. 9, 44 pp.

\bibitem{OPTz}
T.~Oh, O.~Pocovnicu, N.~Tzvetkov,
{\it Probabilistic local Cauchy theory of the cubic nonlinear wave equation in negative Sobolev spaces}, Ann. Inst. Fourier (Grenoble) 72 (2022), no. 2, 771--830.

\bibitem{OWZ}
T.~Oh, Y.~Wang, Y.~Zine,
{\it Three-dimensional stochastic cubic nonlinear wave equation with almost space-time white noise},  Stoch. Partial Differ. Equ. Anal. Comput. (2022). https://doi.org/10.1007/s40072-022-00237-x

\bibitem{Seo}
K.~Seong,
{\it Invariant Gibbs dynamics for the two-dimensional Zakharov-Yukawa system}, arXiv:2111.11195v1 [math.AP].

\bibitem{Sim}
B.~Simon,
{\it The $P(\varphi)_2$ Euclidean (quantum) field theory}, Princeton Series in Physics. Princeton University Press, Princeton, N.J., 1974. xx+392 pp.

\bibitem{Tao}
T.~Tao, 
{\it Nonlinear dispersive equations. Local and global analysis}, CBMS Regional Conference Series in Mathematics, 106. Published for the Conference Board of the Mathematical Sciences, Washington, DC; by the American Mathematical Society, Providence, RI, 2006. xvi+373 pp.

\bibitem{TTz}
L.~Thomann, N.~Tzvetkov,
{\it Gibbs measure for the periodic derivative nonlinear Schr\"odinger equation}, Nonlinearity 23 (2010), no. 11, 2771--2791.


\bibitem{Wang}
Y.~Wang,
{\it Notes on the random average operator}, preprint.


\end{thebibliography}
\end{document}